\documentclass[12pt,a4paper,leqno]{amsart}

\usepackage[top=1.1in, bottom=1.3in, left=1.5in, right=1.5in]{geometry}

\usepackage[ansinew]{inputenc}
\usepackage[T1]{fontenc}
\usepackage[english]{babel}
\usepackage{amsthm}
\usepackage{amsfonts}         
\usepackage{amsmath}
\usepackage{amssymb}
\usepackage{breqn}
\usepackage{bm}
\usepackage{url}
\usepackage{esint}
\usepackage{fancyhdr}

\newcommand{\R}{\mathbb{R}}
\newcommand{\C}{\mathbb{C}}

\newcommand{\N}{\mathbb{N}}

\newcommand{\Z}{\mathbb{Z}}

\newcommand{\A}{\mathcal{A}}
\newcommand{\B}{\mathcal{B}}
\renewcommand{\C}{\mathcal{C}}

\newcommand{\J}{\mathcal{J}}

\newcommand\pprec{\prec\mkern-5mu\prec}
\newcommand{\q}{q}
\newcommand{\s}{s}
\renewcommand{\r}{r}

\newcommand{\U}{\mathcal{U}}
\newcommand{\DD}{\mathcal{D}}
\newcommand{\RR}{\mathcal{R}}
\newcommand{\QQ}{\mathcal{Q}}
\newcommand{\UU}{\mathcal{U}}
\newcommand{\VV}{\mathcal{V}}

\theoremstyle{plain}
\newtheorem{theorem}{Theorem}
\newtheorem{lemma}[theorem]{Lemma}
\newtheorem{prop}[theorem]{Proposition}

\theoremstyle{remark}
\newtheorem{remark}{Remark}

\numberwithin{equation}{section}
\begin{document}

\title{On the largest square divisor of shifted primes}

\author{Jori Merikoski}

\address{Department of Mathematics and Statistics, University of Turku, FI-20014 University of Turku,
Finland}
\email{jori.e.merikoski@utu.fi}

\begin{abstract} We show that there are infinitely many primes $p$ such that $p-1$ is divisible by a square $d^2 \geq p^\theta$ for $\theta=1/2+1/2000.$ This improves the work of Matom\"aki (2009) who obtained the result for $\theta=1/2-\varepsilon$ (with the added constraint that $d$ is also a prime), which improved the result of Baier and Zhao (2006) with $\theta=4/9-\varepsilon.$  Similarly as in the work of Matom\"aki, we apply Harman's sieve method to detect primes $p \equiv 1 \, (d^2)$. To break the $\theta=1/2$ barrier we prove a new bilinear equidistribution estimate modulo smooth square moduli $d^2$ by using a similar argument as Zhang (2014) used to obtain equidistribution beyond the Bombieri-Vinogradov range for primes with respect to smooth moduli. To optimize the argument we incorporate technical refinements from the Polymath project (2014). Since the moduli are squares, the method produces complete exponential sums modulo squares of primes which are estimated using the results of Cochrane and Zheng (2000).
\end{abstract}



\maketitle

\section{Introduction}
A famous open problem in number theory is to show that there are infinitely many prime numbers of the form $p=n^2+1$. As this problem appears hopeless with the current techniques, it is reasonable to consider the easier question of finding primes $p$ such that $p-1$ is divisible by a large square $d^2 \geq p^{\theta}$, $\theta \in (0,1)$. Note that the number of integers up to $X$ which are divisible by a square of size $X^\theta$ is of order $X^{1-\theta/2}.$  

By Linnik's Theorem  on the least prime in arithmetic progressions \cite{linnik} we know that the result is true for some $\theta > 0$. The first result specifically on square divisors is the exponent $\theta = 4/9-\varepsilon$ (for any $\epsilon > 0$) obtained by Baier and Zhao \cite{bzsparse,bzsquare} as an application of their large sieve for sparse sets of moduli. Baier and Zhao interpret the problem as an equidistribution problem for primes $p \equiv 1 \,\,(d^2)$, after which the result follows from their Bombieri-Vinogradov Theorem for sparse sets of moduli \cite[Theorem 3]{bzsparse}.

The current record on this problem is $\theta=1/2-\varepsilon$ (for any $\epsilon >0$) by Matom\"aki \cite{matomaki}. Matom\"aki applies Harman's sieve method with Type II information obtained using the large sieve of Baier and Zhao \cite{bzsquare}. It is noteworthy that in the results of Matom\"aki and Baier \& Zhao the divisor $d^2$ may be restricted to be a square of a prime. The result of Matom\"aki was strengthened by Baker \cite{baker} who obtained the Bombieri-Vinogradov Theorem for square moduli up to $d^2 < X^{1/2-\varepsilon}.$ Note that the exponent $\theta=1/2$ is the limit of what can be obtained assuming the Generalized Riemann Hypothesis.

We obtain for the first time the result with an exponent $\theta > 1/2$.   Our main theorem is
\begin{theorem}\label{maintheorem}
Let $a \neq 0$ be an integer. There are infinitely many primes $p$ such that $d^2 | (p-a)$ for some integer $d$ with
\begin{align*}
d^2 \geq p^{1/2+1/2000}.
\end{align*}
\end{theorem}
Similarly as in the work of Matom\"aki \cite{matomaki}, we employ Harman's sieve. To break the $\theta=1/2$ barrier, we will obtain a new bilinear equidistribution estimate (Proposition \ref{typeii}) by applying a similar technique as in Zhang's work \cite{zhang} on bounded gaps between primes, incorporating the ideas developed in the subsequent Polymath project \cite{polymath}. 

\begin{remark}
We should note that the exponent $1/2000$ in Theorem \ref{maintheorem} has not been fully optimized. By optimizing the sieve argument one should be able to increase this to some exponent between $1/500$ and $1/1000$. We do not pursue this issue here since our primary goal was to obtain the result for some exponent $\theta > 1/2$, and we feel that efforts on improving the result should first be directed at obtaining more arithmetical information (eg. a two dimensional `Type I$_2$ estimate' , or a three dimensional Type III estimate as in \cite{polymath, zhang}) before optimizing the sieve.
\end{remark}

\subsection{Structure of the proof}
In Section \ref{harmansection} we state a more quantitative version of Theorem \ref{maintheorem} and apply Harman's sieve method (cf. Harman's book \cite{harman}, for instance) to give a proof of this. The idea of applying Harman's sieve to this problem goes back to the work of Matom\"aki \cite{matomaki}. Motivated by the work of Zhang \cite{zhang} and  the Polymath project \cite{polymath}, we  restrict to divisors  $d^2$ which are of the form $p_1^2 \cdots p_K^2,$ where each prime $p_j$ is of size $X^{\delta/2}$ for some small $\delta >0$. We show (cf. Theorem \ref{theorem2}) that for almost all $d^2 \asymp X^{1/2+1/2000}$ of this form we have
\begin{align*}
\sum_{\substack{p \sim X \\ p \equiv a \,(d^2)} } 1 \, \gg \, \frac{X}{\phi(d^2) \log X}.
\end{align*}
 
Roughly speaking, Harman's sieve is a combinatorial device of breaking a sum over primes in a set $\A$ into sums of Type I  and Type II:
\begin{align*}
\text{Type I:} \quad  \sum_{ \substack{un \in \A \\ u \sim U}} a_u, \quad \quad \quad \text{Type II:} \quad \sum_{ \substack{uv \in \A \\ u \sim U, \, v \sim V}} a_u b_v,
\end{align*}
 where the coefficients $a_u$ and $b_v$ are arbitrary divisor bounded functions. Because we need the Type II estimate for almost all moduli $d^2$, it can be stated as an averaged bilinear equidistribution estimate (Proposition \ref{typeii}). This is similar to the recent bilinear equidistribution estimates of Zhang \cite[Section 7]{zhang} and the Polymath project \cite[Theorem 5.1]{polymath} concerning the Bombieri-Vinogradov Theorem, with the exception that our moduli run over perfect squares. We note that  \cite{polymath,zhang} use Heath-Brown's identity to obtain a combinatorial decomposition of a sum over primes into different types of sums. 

More precisely, to apply Harman's sieve we require three types of arithmetical information, Type I/II estimate (Proposition \ref{typei}), Type II estimate (Proposition \ref{typeii}) and Type I estimate (Proposition \ref{typeipure}).  Note here a difference in terminology: our Type II range is called  a Type I range in \cite{polymath,zhang}. By Type I/II and Type I we refer to sums where there is already a smooth variable present, which is the terminology used in Harman's book \cite{harman}.

Due to the fact that the set of moduli is very sparse, our Type II information is much narrower than in the situation of \cite{polymath,zhang}. The underlying philosophy for us (as well as in the work of Matom\"aki \cite{matomaki}) is that if one can use certain arithmetical information with Vaughan's or Heath-Brown's identity to give an asymptotic formula, then with more narrow arithmetical information one can still hope to obtain lower and upper bounds by using Harman's sieve. The advantage of Harman's sieve compared to asymptotic sieves is that we can use positivity to regard certain sums as error terms whose contribution  can be bounded numerically. These numerics ultimately determine the exponent $1/2000$ in Theorem \ref{maintheorem}, when combined with the restrictions for the Type II estimate.

In Sections \ref{typeisection} and \ref{typeipuresection} we give proofs of the Type I/II and Type I estimates. These are fairly standard applications of Poisson summation formula and Cauchy-Schwarz, which result in incomplete exponential sums that can be bounded using the P\'olya-Vinogradov method. 

The proof of the Type II estimate takes up most of the paper. In Section \ref{typeiisection} we apply Linnik's dispersion method to reduce the proof to certain incomplete exponential sums which are estimated in Section \ref{expsumsection}. The method is very similar to the argument in \cite[Section 5]{polymath}, especially the proof of  \cite[Theorem 5.1(ii)]{polymath}. Here we will need the fact that $d$ is a product of small primes to obtain a suitable factorization $d=rq$. The idea of using well-factorable moduli goes back to the pioneering work of Fouvry-Iwaniec \cite{fi}, while the idea of using very smooth moduli is due to Zhang \cite{zhang}. The main difference in Section \ref{typeiisection} compared to \cite{polymath,zhang} is that the set of moduli is sparse, which means that the optimization of applications of Cauchy-Schwarz is slightly different.

In Section \ref{expsumsection}, because the moduli are squares, we need to consider incomplete exponential sums of the type
\begin{align*}
 \sum_{n} \psi_N(n) e_{q}(f(n)),
\end{align*}
where $f$ is a rational function, $q$ is cube free, and $\psi_N(n)$ is a smoothing of $1_{n \sim N}$. Similarly as in \cite[Proposition 4.12]{polymath}, we apply Heath-Brown's $q$-van der Corput method to complete the sum. Additionally to the proofs in \cite[Section 4]{polymath}, we need a bound for exponential sums  of the form
\begin{align*}
\sum_{n \in \Z/ p^2 \Z} e_{p^2} (g(n)),
\end{align*}
where $p$ is a prime and $g$ is a rational function. Using the results of Cochrane and Zheng \cite{cz}, we are able to obtain square root cancellation in the generic case for these sums.

\subsection{Notations}
We use the following notations: for functions $f$ and $g$ with $g$ positive, we write $f \ll g$ or $f= \mathcal{O}(g)$ if there is a constant $C$ such that $|f|  \leq C g.$ The notation $f \asymp g$ means $g \ll f \ll g.$ The constant may depend on some parameter, which is indicated in the subscript (e.g. $\ll_{\epsilon}$).
We write $f=o(g)$ if $f/g \to 0$ for large values of the variable. For variables we write $n \sim N$ meaning $N<n \leq 2N$. 

It is convenient for us to define
\begin{align*}
A \pprec B
\end{align*}
to mean $A \, \ll_\epsilon X^{\epsilon} B.$ A typical bound we use is $\tau_k(n) \pprec 1$  for $n \ll X$, where $\tau_k$ is the $k$-fold divisor function. We say that an arithmetic function $f$ is divisor-bounded if $|f(n)| \ll \tau_k(n)$ for some $k$.

We let $\eta >0$ denote a sufficiently small constant, which may be different from place to place. For example, $A \ll X^{-\eta}B$ means that the bound holds for some $\eta >0.$

For a statement $E$ we denote by $1_E$ the characteristic function of that statement. For a set $A$ we use $1_A$ to denote the characteristic function of $A.$ 

We also define $P(w):= \prod_{p\leq w} p,$ where the product is over primes. 

We let $e(x):= e^{2 \pi i x}$ and $e_q(x):= e(x/q)$ for any integer $q \geq 1$. For integers $a,$ $b$, and $q \geq 1$ we define $e_{q}(a/b) := e(a\overline{b}/q)$ if $b$ is invertible modulo $q$, where $\overline{b}$ is the solution to $b\overline{b} \equiv 1 \,\, (q).$ If $b$ is not invertible modulo $q$, we set $e_q(a/b) =0.$ 

\subsection{Acknowledgements}
I am grateful to my supervisor Kaisa Matom\"aki for support and comments. I also express my gratitude to Emmanuel Kowalski for helpful discussions as well as for hospitality during my visit to ETH Z\"urich. I also wish to thank the referee for comments. During the work the author was supported by a grant from the Magnus Ehrnrooth Foundation.

\section{Applying Harman's sieve} \label{harmansection}
 Let $X \gg 1$ and let $\delta >0$ be small. Let $\varpi:=1/4000,$ $D:= X^{1/2+2\varpi}$, $K:=\lceil 1/\delta \rceil$, $P:= D^{1/K}$, and define
\begin{align*}
I_j := (2^{j-1}P^{1/2}, 2^{j} P^{1/2} ] \quad  \text{for} \quad j=1,2,\dots, K.
\end{align*}
We set
\begin{align} \label{ddefinition}
\DD := \{p_1^2 p_2^2 \cdots p_K^2: \, p_j \in I_j \quad \text{for} \quad j=1,2,\dots,K \},
\end{align}
so that $d^2 \in \DD$ is of size $\asymp D$ and is a square of a squarefree integer.

Fix an integer $a \neq 0$. Fix also a $C^\infty$-smooth  function $ 0 \leq \psi \leq 1$, supported on the interval $[1,2]$ and satisfying $\psi(x)=1 $ for $1+\eta \leq x \leq 2-\eta$ for some sufficiently small $\eta >0$.  For $d^2 \in \DD$ and $z < X$, denote
\begin{align*}
S(\A^d,z):= \sum_{\substack{n \equiv a \,\, (d^2) \\ (n,P(z))=1}} \psi(n/X) \quad \text{and} \quad  S(\B^d,z):=  \frac{1}{\phi(d^2)}\sum_{\substack{(n,d^2)=1 \\ (n,P(z))=1}} \psi(n/X),
\end{align*}
so that $S(\A^d,2\sqrt{X})$ is a sum over primes $p \equiv a \,\, (d^2)$ of size $p\asymp X$. Then Theorem \ref{maintheorem} follows from

\begin{theorem} \label{theorem2}
Let $\DD$ be as in (\ref{ddefinition}). Then there exists  $\delta, \eta >0$ such that for all but $\mathcal{O}(D^{1/2} X^{-\eta})$ of the moduli $d^2  \in \DD$ we have
\begin{align*}
S(\A^d,2\sqrt{X}) > 0.05 \cdot S(\B^d,2\sqrt{X}).
\end{align*}
\end{theorem}

The proof of this is given at the end of this section, by applying Harman's sieve method. For the sieve we need arithmetical information given by the following  propositions. To state these propositions, let us define the Type II parameter $\sigma := 1/19.5.$
\begin{prop} \emph{\textbf{(Type I/II estimate).}} \label{typei} Let $\DD$ be as in (\ref{ddefinition}), $\sigma=1/19.5,$ and $LMN = X,$ where $L,M,N \geq 1$, $M \leq X^{1/2-\sigma},$ and $N \leq X^{1/8 +\sigma/2 -5\varpi /2 - \eta}.$ Let $\alpha(m)$ and $\beta(n)$ be divisor-bounded functions. Then for all $d^2 \in \DD$
\begin{align*}
\bigg | \sum_{\substack{\ell mn \equiv a \, (d^2) \\ m \sim M, \, \, n \sim N}} \alpha(m) \beta(n) \psi(\ell mn/X) - \frac{1}{\phi(d^2)} \sum_{\substack{( \ell mn,d^2)=1 \\ m \sim M, \, \, n \sim N}} \alpha(m) \beta(n) \psi(\ell mn/X) \bigg | \, \ll \frac{X^{1-\eta}}{D}.
\end{align*}

\end{prop}

\begin{prop} \emph{\textbf{(Type II estimate).}} \label{typeii} Let $\DD$ be as in (\ref{ddefinition}), $\sigma=1/19.5$, and $MN = X$ with
\begin{align*}
M, N \in [X^{1/2-\sigma}, X^{1/2+\sigma}] \setminus [X^{1/2-2\varpi - \delta}, X^{1/2 + 2 \varpi + \delta}]
\end{align*}
Let $\alpha(m)$ and $\beta(n)$ be divisor-bounded functions. Then
\begin{align*}
\sum_{d^2 \in \DD} \bigg | \sum_{\substack{mn \equiv a \, (d^2) \\ m \sim M, \, \, n \sim N}} \alpha(m) \beta(n) \psi(mn/X) - \frac{1}{\phi(d^2)} \sum_{\substack{(mn,d^2)=1 \\ m \sim M, \, \, n \sim N}} \alpha(m) \beta(n) \psi(mn/X) \bigg | \, \ll \frac{X^{1-\eta}}{\sqrt{D}}.
\end{align*}
\end{prop}

\begin{remark} The width of the Type II range is mainly determined by the condition (cf. (\ref{sigma}) below)
\begin{align*}
19 \sigma + 90 \varpi + 71 \delta < 1,
\end{align*}
which certainly holds for $\sigma = 1/19.5$ and $\varpi =1/4000$ for some $\delta>0.$ That is, we get a positive $\varpi$ as soon as $\sigma < 1/19.$ Compare this to \cite[Theorem 5.1(ii)]{polymath} which gives a positive $\varpi$ if $\sigma < 1/4.$ This difference in quality is solely due to the fact that we work with a sparse set of moduli. To maximize the size of $\varpi,$ we want to make $\sigma$ as small as possible. The smallest admissible value of $\sigma$ is determined by numerical computations applying Harman's sieve method below. By optimizing the sieve argument carefully we could work with a slightly smaller value of $\sigma$ which would allow us to take somewhat larger $\varpi$.
\end{remark}

\begin{remark} In principle there is nothing particular about the moduli being perfect squares; with similar arguments one should be able to obtain Type II information for other classes of well-factorable sparse moduli. However, the quality of the estimate decreases rapidly as the density of the moduli set decreases.
\end{remark}

Note that we have a gap $[X^{1/2-2\varpi-\delta},X^{1/2+2\varpi+\delta}]$ in the Type II information. This is due to the fact that the moduli run over a sparse set (cf. Remark \ref{gapremark} at the end of Section \ref{typeiisection}). To compensate for this we require
\begin{prop}\emph{\textbf{(Type I estimate).}} \label{typeipure} Let $\DD$ be as in (\ref{ddefinition}) and let $MN = X$ with
\begin{align*}
M \leq X^{1/2 + 2 \varpi + \delta}.
\end{align*}
Let $\alpha(m)$  be a divisor-bounded function. Then
\begin{align*}
\sum_{d^2 \in \DD} \bigg | \sum_{\substack{mn \equiv a \, (d^2) \\ m \sim M }} \alpha(m)  \psi(mn/X) - \frac{1}{\phi(d^2)} \sum_{\substack{(mn,d^2)=1 \\ m \sim M }} \alpha(m)  \psi(mn/X) \bigg | \, \ll \frac{X^{1-\eta}}{\sqrt{D}}.
\end{align*}
\end{prop}

We will prove these three estimates in increasing order of difficulty, Type I/II estimate being the easiest, and Type II being much harder than the other two. In the remainder of this section we apply 
these estimates to give a proof of Theorem \ref{theorem2}.

We combine the first two propositions to get
\begin{prop} \label{funprop} Let $\DD$ be as in (\ref{ddefinition}). Let $U,V \geq 1,$ $U \leq X^{1/2-\sigma}$, $V \leq  X^{1/8 +\sigma/2 -5\varpi /2 - \eta},$ and let $a_u,b_v$ be divisor bounded coefficients. Let $Z= X^{\sigma -2 \varpi - \delta}.$ Then for all but $\mathcal{O}(D^{1/2} X^{-\eta})$ of $d^2  \in \DD$ we have
\begin{align*}
\sum_{\substack{u \sim U \\ v \sim V}} a_u b_v S(\A^d_{uv},Z) = \sum_{\substack{u \sim U \\ v \sim V}} a_u b_v S(\B^d_{uv},Z) + \mathcal{O}\bigg( \frac{X^{1-\eta}}{D} \bigg).
\end{align*}
\end{prop}
\begin{proof}
The left-hand side is (by using the M\"obius function to detect $(n,P(Z))=1$)
\begin{align*}
& \sum_{\substack{uvn \equiv a \, (d^2) \\ u \sim U, \, \, v \sim V}} a_u b_v \psi(uvn/X) 1_{(n,P(Z))=1}  =  \sum_{\substack{uvek \equiv a \, (d^2) \\ u \sim U, \, \, v \sim V \\ e \, | P(Z)}} a_u b_v \mu(e) \psi(uvek/X) \\
 &  = \sum_{\substack{uvek \equiv a \, (d^2) \\ u \sim U, \, \, v \sim V \\ e \, | P(Z) \\ eu \leq X^{1/2-\sigma}}} a_u b_v \mu(e) \psi(uvek/X) + \sum_{\substack{uvek \equiv a \, (d^2) \\ u \sim U, \, \, v \sim V \\ e \, | P(Z) \\ eu > X^{1/2-\sigma}}} a_u b_v \mu(e) \psi(uvek/X)=: \Sigma_I(\A^d) + \Sigma_{II}(\A^d).
\end{align*}
We split the sum over $\B^d$ similarly into $\Sigma_I(\B^d) + \Sigma_{II}(\B^d).$
For $\Sigma_I$, by Proposition \ref{typei} we have
\begin{align*}
\Sigma_I(\A^d)= \Sigma_I(\B^d) + \mathcal{O}\bigg( \frac{X^{1-\eta}}{D} \bigg),
\end{align*}
if we combine variables $m=eu$, relabel $n=v,$ and split the sums dyadically.

In $\Sigma_{II}$ we note that since $eu > X^{1/2-\sigma}$, $U \leq x^{1/2-\sigma}$ and $e \, | P(Z),$ we can use the greedy algorithm to partition  the sum (writing $e=q_1\cdots q_\ell$ for primes $q_1 < \cdots < q_\ell  \leq Z$)
\begin{align*}
\Sigma_{II}(\A^d) &= \sum_{\substack{uvek \equiv a \, (d^2) \\ u \sim U, \, \, v \sim V \\ e \, | P(Z) \\ eu > X^{1/2-\sigma}}}  a_u b_v \mu(e) \psi(uvek/X) \\
 & = \sum_{\ell \, \ll \log X} (-1)^\ell \sum_{j \leq \ell } \sum_{u,v}  \hspace{-5pt}  \sum_{\substack{q_1 < q_2 < \cdots < q_\ell \leq Z \\  uvq_1 \cdots q_\ell k \equiv a \, (d^2)\\u q_1 q_2 \cdots q_j \in [X^{1/2-\sigma},X^{1/2-2\varpi -\delta} ] \\ u q_1 q_2 \cdots q_{j-1}  < X^{1/2-\sigma}}} \hspace{-25pt} a_u b_v  \psi(uvq_1\cdots q_\ell k/X),
\end{align*}
and similarly for  $\Sigma_{II}(\B^d)$. After writing $m= u  q_1 q_2 \cdots q_j$, $n= k v q_{j+1} q_{j+2} \cdots q_{\ell},$ and removing the cross-condition $q_{j+1 } > q_j$ by Perron's formula (cf. \cite[Chapter 3]{harman}, for instance), we obtain from Proposition \ref{typeii}  that
\begin{align*}
\Sigma_{II}(\A^d)= \Sigma_{II}(\B^d) + \mathcal{O}\bigg( \frac{X^{1-\eta}}{D} \bigg)
\end{align*}
for all but $\mathcal{O}(D^{1/2} X^{-\eta})$ of $d^2 \in \DD$.
\end{proof}

We also require the following variant of the above proposition, obtained by applying Propositions  \ref{typeii} and \ref{typeipure}.
\begin{prop} \label{funprop2} Let $\DD$ be as in (\ref{ddefinition}). Let $U \geq 1,$ $U \leq X^{1/2+ 2\varpi + \delta}$,  and let $a_u$ be a  divisor bounded coefficient. Let $Z= X^{\sigma -2 \varpi - \delta}.$ Then for all but $\mathcal{O}(D^{1/2} X^{-\eta})$ of $d^2  \in \DD$ we have
\begin{align*}
\sum_{\substack{u \sim U }} a_u S(\A^d_{u},Z) = \sum_{\substack{u \sim U}} a_u  S(\B^d_{u},Z) + \mathcal{O}\bigg( \frac{X^{1-\eta}}{D} \bigg).
\end{align*}
\end{prop}
\begin{proof}
The left-hand side is
\begin{align*}
& \sum_{\substack{un \equiv a \, (d^2) \\ u \sim U}} a_u  \psi(un/X) 1_{(n,P(Z))=1}  =  \sum_{\substack{uen \equiv a \, (d^2) \\ u \sim U, \\ e \, | P(Z)}} a_u  \mu(e) \psi(uen/X) \\
 &  = \sum_{\substack{uen \equiv a \, (d^2) \\ u \sim U, \\ e \, | P(Z) \\ eu \leq X^{1/2+2 \varpi + \delta} }} a_u  \mu(e) \psi(uen/X) +  \sum_{\substack{uen \equiv a \, (d^2) \\ u \sim U, \\ e \, | P(Z) \\ eu > X^{1/2+2 \varpi + \delta} }} a_u  \mu(e) \psi(uen/X) =: \Sigma_I(\A^d) + \Sigma_{II}(\A^d).
\end{align*}
We split the sum over $\B^d$ similarly into $\Sigma_I(\B^d) + \Sigma_{II}(\B^d).$ The rest of the argument is essentially the same as in the proof of Proposition \ref{funprop}, using Proposition \ref{typeipure} instead of Proposition \ref{typei} to handle $\Sigma_I$. In $\Sigma_{II}$ we split the sums into $e=q_1 \cdots q_\ell$ with $uq_1 \cdots q_j \in  [X^{1/2+2 \varpi +\delta},X^{1/2+\sigma} ]$ and $ u q_1 q_2 \cdots q_{j-1} < X^{1/2+2 \varpi +\delta}$ for some $j \leq \ell$ to obtain Type II sums. 
\end{proof}

\subsection{Buchstab decompositions}
In this section we give the proof of Theorem \ref{theorem2}. All estimates given in this section are interpreted as holding for all but $\mathcal{O}(D^{1/2} X^{-\eta})$ moduli $d^2 \in \DD$.

The general idea of Harman's sieve is to use Buchstab's identity to decompose the sum $S(\C^d,2\sqrt{X})$  (in parallel for $\C^d=\A^d$ and $\C^d=\B^d$) into a sum of the form $\sum_k \epsilon_k S_k(\C^d),$ where $\epsilon_k \in \{-1,1\},$ and  $S_k(\C^d) \geq 0$ are sums over almost-primes.  Since we are interested in a lower bound, for $\C^d=\A^d$ we can insert the trivial estimate $S_k(\A^d) \geq 0$ for any $k$ such that the sign $\epsilon_k =1;$ these sums are said to be discarded. For the remaining $k$ we will obtain an asymptotic formula by using Propositions \ref{typeii}, \ref{funprop}, and \ref{funprop2}. That is, if $\mathcal{K}$ is the set of indices that are discarded, then
\begin{align*}
S(\A^d, 2\sqrt{X})&= \sum_k \epsilon_k S_k(\A^d) \geq \sum_{k \notin \mathcal{K}} \epsilon_k S_k(\A^d)  \\
&\sim \sum_{k \notin \mathcal{K}} \epsilon_k  S_k(\B^d) =  S(\B^d,2\sqrt{X}) -   \sum_{k \in \mathcal{K}}  S_k(\B^d).
\end{align*} 
We are successful if we can then show that $\sum_{k \in \mathcal{K}} S_k(\B^d) \leq (1-\mathfrak{C}(\sigma)) S(\B^d, 2\sqrt{X})$ for some $\mathfrak{C}(\sigma) > 0.$ Obtaining this ultimately determines the smallest admissible exponent $\sigma$ (as $\mathfrak{C}(\sigma)$ is a decreasing function of $\sigma$), which in turn determines the exponent $\varpi$. 

To bound these error terms we need a lemma which converts sums over almost primes into integrals which can be bounded numerically.  Let $\omega(u)$ denote the Buchstab function (cf. \cite[Chapter 1]{harman} for the properties below, for instance), so that by the Prime Number Theorem for $Y^{\epsilon} < z < Y$ 
\begin{align} \label{buchasymp}
\sum_{Y< n \leq 2Y} 1_{(n,P(z))=1} = (1+o(1)) \omega \left(\frac{\log Y}{\log z} \right) \frac{Y}{\log z}.
\end{align}
 Note that for $1< u \leq 2$ we have $\omega(u)=1/u.$ In the numerical computations we will use the following upper bound for the Buchstab function (cf. \cite[Lemma 5]{hbjia}, for instance)
\begin{align*}
\omega(u) \, \leq \begin{cases} 0, &u < 1 \\
1/u, & 1 \leq u < 2 \\
(1+\log(u-1))/u, &2 \leq u < 3 \\
0.5644, &  3 \leq u < 4 \\
0.5617, & u \geq 4.
\end{cases}
\end{align*}
In the lemma below we assume that the range $\mathcal{U}\subset [X^{2 \delta},X]^{k}$ is sufficiently well-behaved, e.g. an intersection of sets of the type $\{ \boldsymbol{u}: u_i < u_j \}$ or $\{\boldsymbol{u}: V <  f(u_1, \dots,u_k) < W\}$ for some polynomial $f$ and some fixed $V,W.$

\begin{lemma} \label{bilemma} Let $\mathcal{U} \subset [X^{2 \delta},X]^{k}.$ Then
\begin{align*}
\sum_{(p_1, \dots , p_k) \in \mathcal{U}} S(\B^d_{p_1, \dots, p_k},p_k) = S(\B^d,2\sqrt{X})(1+\mathcal{O}(\eta))\int \omega (\boldsymbol{\alpha }) \frac{d\alpha_1 \cdots d\alpha_k}{\alpha_1\cdots\alpha_{k-1}\alpha_k^2},
\end{align*}
where the integral is over the range 
\begin{align*}
\{\boldsymbol{\alpha}: \, (X^{\alpha_1}, \dots, X^{\alpha_k}) \in \mathcal{U}\}
\end{align*}
 and $\omega(\boldsymbol{\alpha})= \omega(\alpha_1,\dots,\alpha_k):= \omega((1-\alpha_1-\cdots 
 -\alpha_k)/\alpha_k)$.
\end{lemma}
\begin{proof}
By the definition of $\psi$, by (\ref{buchasymp}), and by the Prime Number Theorem, the left-hand side is
\begin{align*}
&\frac{1}{\phi(d^2)} \sum_{(p_1, \dots , p_k) \in \mathcal{U}} \sum_q 1_{(q,P(p_k))=1} \psi(p_1 \cdots p_k q /X) \\
&= (1+\mathcal{O}(\eta))\frac{X}{\phi(d^2)} \sum_{(p_1, \dots , p_k) \in \mathcal{U}} \frac{1}{p_1\cdots p_k \log p_k} \omega \left( \frac{\log(X/(p_1\cdots p_k))}{\log p_k} \right) \\
&= (1+\mathcal{O}(\eta))\frac{X}{\phi(d^2)} \hspace{-5pt}  \sum_{(n_1,\dots,n_k ) \in \mathcal{U}} \frac{1}{n_1\cdots n_k (\log n_1) \dots (\log n_{k-1} )\log^2 n_k} \omega \left( \frac{\log(X /(n_1\cdots n_k))}{\log n_k} \right) \\
&= (1+\mathcal{O}(\eta))\frac{X}{\phi(d^2)}  \int_{\mathcal{U}}  \omega \left( \frac{\log(X/(u_1\cdots u_k))}{\log u_k} \right)   \frac{du_1\cdots du_k}{u_1\cdots u_k (\log u_1) \dots (\log u_{k-1} )\log^2 u_k}\\
&= \frac{(1+\mathcal{O}(\eta))X}{\phi(d^2)\log x}  \int \omega (\boldsymbol{\alpha }) \frac{d\alpha_1 \cdots d\alpha_k}{\alpha_1\cdots\alpha_{k-1}\alpha_k^2} =(1+\mathcal{O}(\eta))S(\B^d,2\sqrt{X}) \int \omega (\boldsymbol{\alpha }) \frac{d\alpha_1 \cdots d\alpha_k}{\alpha_1\cdots\alpha_{k-1}\alpha_k^2}
\end{align*}
by the change of variables $u_j=X^{\alpha_j}$.
\end{proof}
\begin{remark} The factor $\int \omega(\boldsymbol{\alpha}) \frac{ d \alpha_1\cdots d\alpha_k }{\alpha_1\cdots\alpha_{k-1}\alpha_k^2}$ is called the deficiency of the corresponding sum. By the above lemma it is up to a factor of $(1+\mathcal{O}(\eta))$ the ratio of the sum to $S(\B^d,2\sqrt{X})$.
\end{remark}

 We are now ready to prove Theorem \ref{theorem2}: Let $Z= X^{\sigma -2 \varpi - \delta}$ and $\alpha:= 1/8 + \sigma/2- 5 \varpi/2 -\eta.$ For $\C^d\in \{\A^d,\B^d\}$, we use Buchstab's identity twice to get
\begin{align*}
S(\C^d,2\sqrt{X}) = S(\C^d,Z) - \sum_{Z < p \leq 2\sqrt{X}} S(\C^d_p, Z) + \sum_{\substack{Z < p_2 < p_1 \leq  2\sqrt{X} \\ p_1p_2^2 \leq X}} S(\C^d_{p_1 p_2}, p_2).
\end{align*}

By Proposition \ref{funprop}, we have
\begin{align*}
S(\A^d,Z) = S(\B^d,Z) + \mathcal{O}(X^{1-\eta}/D )
\end{align*}
and
\begin{align*}
\sum_{Z < p \leq X^{1/2-\sigma}} S(\A^d_p, Z) = \sum_{Z < p \leq X^{1/2-\sigma}} S(\B^d_p, Z) + \mathcal{O}(X^{1-\eta}/D ).
\end{align*}
(recall our convention that in this section all estimates are interpreted as holding for all but $\mathcal{O}(D^{1/2} X^{-\eta})$ of $d$ with $d^2 \in \DD$). By Proposition  \ref{funprop2} we obtain
\begin{align*}
\sum_{ X^{1/2-\sigma} < p \leq  2\sqrt{X} } S(\A^d_p, Z) = \sum_{X^{1/2-\sigma} < p \leq  2\sqrt{X} } S(\B^d_p, Z) + \mathcal{O}(X^{1-\eta}/D ),
\end{align*}
so that in  the first two sums we get asymptotic formulas.

For the third sum we write
\begin{align*}
 \sum_{\substack{Z < p_2 < p_1 \leq 2\sqrt{X} \\ p_1p_2^2 \leq X}} S(\C^d_{p_1 p_2}, p_2) = S_1(\C^d) +  S_2(\C^d) + S_3(\C^d),
\end{align*}
where
\begin{align*}
S_1(\C^d):=  \sum_{\substack{Z < p_2 < p_1 \leq  2\sqrt{X} \\ p_1p_2 \leq X^{1/2-\sigma} \\ p_2 \leq X^\alpha \,\, \text{or} \,\, p_1 p_2^2 \leq X^{1/2+\sigma}}} & S(\C^d_{p_1 p_2}, p_2), \quad \quad  
S_2(\C^d):=  \sum_{\substack{Z < p_2 < p_1 \leq  2\sqrt{X} \\ p_1p_2 \leq X^{1/2-\sigma} \\ p_2 > X^\alpha\,\, \text{and} \,\, p_1 p_2^2 > X^{1/2+\sigma}}} S(\C^d_{p_1 p_2}, p_2), \\ \text{and} \quad \quad \quad
& S_3(\C^d):=  \sum_{\substack{Z < p_2 < p_1 \leq  2\sqrt{X} \\ p_1p_2 > X^{1/2-\sigma} \\ p_1 p_2^2 \leq X}} S(\C^d_{p_1 p_2}, p_2).
\end{align*}

\subsubsection{Sum $S_1(\C^d)$} Let $\U(i)$ denote the range for the sum $S_{i}(\C^d)$ (similarly for $\U(i,j)$ and $\U(i,j,k)$ below).
We apply Buchstab's identity twice to get
\begin{align*}
S_{1}(\C^d) &= \sum_{(p_1,p_2) \in \U(1)} S(\C^d_{p_1 p_2}, Z) - \sum_{ \substack{(p_1,p_2) \in \U(1) \\ Z < p_3 < p_2 \\ p_1p_2p_3^2 \leq X }} S(\C^d_{p_1 p_2p_3}, Z) +\sum_{ \substack{(p_1,p_2) \in \U(1) \\ Z <p_4 < p_3 < p_2 \\ p_1p_2p_3^2 \leq X, \, \,  p_1p_2p_3p_4^2 \leq X}} S(\C^d_{p_1 p_2p_3 p_4}, p_4) \\
& =: S_{1,1}(\C^d) -S_{1,2}(\C^d) +S_{1,3}(\C^d).
\end{align*}

In the first sum we get an asymptotic formula by Proposition \ref{funprop}. In the second sum, we get an asymptotic formula by Proposition \ref{funprop2} if $p_2 \leq X^\alpha$, since then $p_3 \leq X^\alpha$ and $p_1p_2 \leq X^{1/2- \sigma}$. In the remaining part we have $p_1 p_2p_3 < p_1p_2^2 \leq X^{1/2+\sigma},$ so that we get an asymptotic formula by applying Propositions \ref{typeii} and \ref{funprop2}.

For the third sum we write $S_{1,3}(\C^d)=S_{1,3,1}(\C^d)+S_{1,3,2}(\C^d),$ where
\begin{align*}
S_{1,3,1}(\C^d) :=  \hspace{-15pt} \sum_{\substack{(p_1,p_2,p_3,p_4) \in \U(1,3) \\ p_1p_2p_3p_4 \leq X^{1/2-\sigma}}} S(\C^d_{p_1 p_2p_3 p_4}, p_4) \quad \text{and} \quad S_{1,3,2}(\C^d):= \hspace{-15pt} \sum_{\substack{(p_1,p_2,p_3,p_4) \in \U(1,3) \\ p_1p_2p_3p_4 > X^{1/2-\sigma}}} S(\C^d_{p_1 p_2p_3 p_4}, p_4).
\end{align*}

\textbf{Sum $S_{1,3,1}(\C^d)$.} We apply Buchstab's identity twice to obtain
\begin{align*}
S_{1,3,1}(\C^d) = \sum_{\substack{(p_1,p_2,p_3,p_4) \in \U(1,3,1)}} & S(\C^d_{p_1 p_2p_3 p_4}, Z) - \sum_{\substack{(p_1,p_2,p_3,p_4) \in \U(1,3,1) \\ Z < p_5 < p_4 \\ p_1p_2p_3p_4p_5^2 \leq X}} S(\C^d_{p_1 p_2p_3 p_4 p_5}, Z)  \\
& + \sum_{\substack{(p_1,p_2,p_3,p_4) \in \U(1,3,1) \\ Z < p_6 < p_5 < p_4 \\ p_1p_2p_3p_4p_5^2 \leq X \\ p_1p_2p_3p_4p_5 p_6^2 \leq X }} S(\C^d_{p_1 p_2p_3 p_4 p_5 p_6}, p_6)  
\end{align*}
Since $p_1p_2p_3p_4 \leq X^{1/2-\sigma},$ we have $p_5 < p_4 < (p_1p_2p_3p_4)^{1/4} < X^{\alpha}.$ Hence, the first two sums have asymptotic formulas by Proposition \ref{funprop}.  In the third sum we apply Proposition \ref{typeii} to the parts where a combination of the variables is in the Type II range and discard the rest, which gives us by Lemma \ref{bilemma} a deficiency
\begin{align*}
 \mathcal{O}(\delta) +  \int f_{1,3,1} (\bm{\alpha}) \omega( \bm{\alpha}) \frac{d \alpha_1d \alpha_2d \alpha_3d \alpha_4d \alpha_5d \alpha_6 }{\alpha_1 \alpha_2 \alpha_3 \alpha_4 \alpha_5 \alpha_6^2} < 0.0095,
\end{align*}
where $f_{1,3,1}$ is the characteristic function of the six dimensional set
\begin{align*}
\{(\bm{\alpha}):  \sigma & - 2 \varpi < \alpha_6 <  \cdots < \alpha_1 < 1/2 - \sigma, \quad  \alpha_1 + \alpha_2 + \alpha_3 +\alpha_4 \leq 1/2 -\sigma, \\
& 2\alpha_2 \leq \max\{ 2\alpha, 1/2+ \sigma - \alpha_1 \}, \quad \max_{k \leq 6} \bigg\{ \alpha_k + \sum_{j \leq k} \alpha_j  \bigg \} \leq 1, \\
& \hspace{200pt}\sum_{j \in I} \alpha_j \notin \J  \,\,\text{for all} \,\, I \subseteq \{1,\dots,6\}  \},
\end{align*}
where $\J := [1/2-\sigma,1/2+\sigma] \setminus [1/2-2 \varpi,1/2+2 \varpi] $. (For the codes used to bound the integrals, see the codepad links at the end of this section).

\textbf{Sum $S_{1,3,2}(\C^d)$.}  We first apply Buchstab's identity upwards to get
\begin{align*}
S_{1,3,2}(\C^d) =   \sum_{\substack{ (p_1,p_2,p_3,p_4) \in \U(1,3,2) }}  S(\C^d_{p_1 p_2p_3p_4}, 2\sqrt{X/p_1p_2p_3p_4}) + \sum_{\substack{ (p_1,p_2,p_3,p_4) \in \U(1,3,2) \\ p_4 < p_5 \leq 2\sqrt{X/p_1p_2p_3p_4} }}  S(\C^d_{p_1 p_2 p_3p_4p_5}, p_5)
\end{align*}
so that in the first sum the implicit variable runs over primes. We apply Proposition \ref{typeii} when we have a variable in the Type II range and discard the rest, which leaves us with  deficiencies (cf. Lemma \ref{bilemma})
\begin{align*}
 \mathcal{O}(\delta) +  \int f_{1,3,2} (\bm{\alpha}) \frac{d \alpha_1d \alpha_2d \alpha_3d \alpha_4 }{(1-\alpha_1-\alpha_2-\alpha_3-\alpha_4)\alpha_1 \alpha_2 \alpha_3  \alpha_4} < 0.016,
\end{align*}
and
\begin{align*}
 \mathcal{O}(\delta) +  \int g_{1,3,2} (\bm{\alpha}) \omega( \bm{\alpha}) \frac{d \alpha_1d \alpha_2d \alpha_3d \alpha_4 d \alpha_5 }{\alpha_1 \alpha_2 \alpha_3  \alpha_4 \alpha_5^2} < 0.0038,
\end{align*}
where  $f_{1,3,2}$ is the characteristic function of the four dimensional set
\begin{align*}
\mathcal{V}:= \{\bm{\alpha}: & \, \sigma- 2 \varpi < \alpha_4 < \alpha_3 < \alpha_2 < \alpha_1 < 1/2 - \sigma, \quad  \alpha_1 + \alpha_2 \leq 1/2 -\sigma, \\
&  2\alpha_2 \leq \max\{ 2\alpha, 1/2+ \sigma - \alpha_1 \}, \quad \max_{k \leq 4} \bigg\{ \alpha_k + \sum_{j \leq k} \alpha_j  \bigg \} \leq 1,
\\ &  \alpha_1 + \alpha_2  + \alpha_3 + \alpha_4 \geq 1/2 - 2 \varpi, \quad \sum_{j \in I} \alpha_j \notin \J  \,\,\text{for all} \,\, I \subseteq \{1,2,3,4\}  \},
\end{align*}
and $g_{1,3,2}$ is the characteristic function of the five dimensional set
\begin{align*}
\{ \bm{\alpha}: (\alpha_1,\alpha_2,\alpha_3,\alpha_4) \in \mathcal{V}, \quad \alpha_4 < \alpha_5 \leq &(1- \alpha_1 -\alpha_2 -\alpha_3-\alpha_4)/2, \\
& \sum_{j \in I} \alpha_j \notin \J  \,\,\text{for all} \,\, I \subseteq \{1,2,3,4,5\}   \}
\end{align*}

\textbf{Total deficiency for $S_1(\C^d)$.} The total deficiency is $< 0.0095+0.016+0.0038 = 0.0293$, so that 
\begin{align*}
S_{1}(\A^d) \geq S_{1}(\B^d) - 0.0293 \cdot S(\B^d,2\sqrt{X}).
\end{align*}

\subsubsection{Sum $S_2(\C^d)$}
Here we apply Buchstab's identity upwards to get
\begin{align*}
S_{2}(\C^d) & =  \sum_{\substack{ (p_1,p_2) \in \U(2) }} S(\C^d_{p_1 p_2}, p_2) \\
&=  \sum_{\substack{ (p_1,p_2) \in \U(2) }}  S(\C^d_{p_1 p_2}, 2\sqrt{X/p_1p_2}) + \sum_{\substack{ (p_1,p_2) \in \U(2) \\ p_2 < p_3 \leq 2\sqrt{X/p_1p_2} }}  S(\C^d_{p_1 p_2 p_3}, p_3),
\end{align*}
so that in the first sum the implicit variable runs over primes. In the second sum we use Proposition \ref{typeii} to handle parts where some combination of variables is in the Type II range and discard the rest. This leads to deficiencies
\begin{align*}
\mathcal{O}(\delta) +  \int f_{2}(\bm{\alpha}) \frac{d \alpha_1 d \alpha_2 }{(1-\alpha_1 - \alpha_2)\alpha_1 \alpha_2 }  < 0.155
\end{align*}
and 
\begin{align*}
\mathcal{O}(\delta) +  \int g_{2}(\bm{\alpha}) \omega(\bm{\alpha})\frac{d \alpha_1 d \alpha_2 d \alpha_3}{\alpha_1 \alpha_2 \alpha_3^2}  < 0.0456.
\end{align*}
Here $f_{2}$ is the characteristic function of the two dimensional set
\begin{align*}
\mathcal{W} := \{ \bm{\alpha}: \sigma - 2 \varpi < \alpha_2 < \alpha_1, \quad \alpha_1+\alpha_2 \leq 1/2-\sigma, \quad \alpha_2 >\alpha, \quad \alpha_1 +2 \alpha_2 > 1/2 + \sigma \},
\end{align*}
and $g_{2}$ is the characteristic function of the three dimensional set
\begin{align*}
\{ \bm{\alpha}: (\alpha_1,\alpha_2) \in \mathcal{W}, \quad \alpha_2 < \alpha_3 \leq (1- \alpha_1 -\alpha_2)/2, \quad \alpha_1 + \alpha_3 \notin \J,  \quad \alpha_2 + \alpha_3 \notin \J  \}
\end{align*}
\textbf{Total deficiency for $S_2(\C^d)$.}  The total deficiency is $< 0.155+0.0456=0.2006$, so that 
\begin{align*}
S_{2}(\A^d) \geq S_{2}(\B^d) - 0.2006 \cdot S(\B^d,2\sqrt{X}).
\end{align*}

\subsubsection{Sum $S_3(\C^d)$}
Here we get an asymptotic formula by Proposition \ref{typeii} if $p_1p_2$ is in the Type II range. We discard the rest, which gives a deficiency
\begin{align*}
\mathcal{O}(\delta) +  \int f_3(\bm{\alpha}) \omega(\bm{\alpha}) \frac{d \alpha_1 d \alpha_2}{\alpha_1 \alpha_2^2} < 0.71153,
\end{align*}
where $f_3$ is the characteristic function of the two dimensional set
\begin{align*}
\{\bm{\alpha}: \sigma -2 \varpi < \alpha_2 < \alpha_1 < 1/2, \quad \alpha_1 + \alpha_2 > 1/2 - 2 \varpi, \quad \alpha_1 + 2\alpha_2 < 1, \quad \alpha_1, \alpha_1 + \alpha_2 \notin \J \}.
\end{align*}
Hence,
\begin{align*}
S_{3}(\A^d) \geq S_{3}(\B^d) - 0.71153 \cdot S(\B^d,2\sqrt{X}).
\end{align*}

\begin{remark} Here we could make the deficiency smaller by applying the role reversal technique to the part where $p_2 \leq x^\alpha$ (cf. for instance \cite[Chapter 5.3]{harman}). However, since in most parts of the discarded sum we have $p_1 \leq X^{1/2-\sigma}$ and $p_1p_2 > X^{1/2+\sigma}$, this is already a narrow range.
\end{remark}

\subsubsection{Proof of Theorem \ref{theorem2}}
By combining the above, we find that for all but $\mathcal{O}(D^{1/2} X^{-\eta})$ moduli $d_2 \in \DD$ we have
\begin{align*}
S(\A^d,2\sqrt{X}) &= S(\A^d,Z) - \sum_{Z < p \leq 2\sqrt{X}} S(\A^d_p, Z) +   S_1(\A^d) +  S_2(\A^d) + S_3(\A^d) \\
&\geq S(\B^d,Z) - \sum_{Z < p \leq 2\sqrt{X}} S(\B^d_p, Z) +   S_1(\B^d) +  S_2(\B^d) + S_3(\B^d)  \\
& \hspace{150pt} - (0.0293  + 0.2006 + 0.71153) \cdot S(\B^d,2\sqrt{X}) \\
&=(1-0.0293  - 0.2006 -0.71153)\cdot S(\B^d,2\sqrt{X}) >  0.05 \cdot S(\B^d,2\sqrt{X}).
\end{align*}
\qed
\begin{remark} If we did not have a gap in the Type II information, that is, if we had Proposition \ref{typeii} for $X^{1/2-\sigma} \ll M,N \ll X^{1/2+\sigma},$ we could apply the buchstab identity with $Z=X^{2\sigma}$ in the argument. However, the gap causes only technical difficulties for small $\varpi$ since Harman's sieve method should depend continuously on the quality of the arithmetical information. To further justify this, note that if we have variables in the range $[X^{\sigma - 2\varpi - \delta},X^{2 \sigma}],$ then we can reduce the deficiency by further applications of Buchstab's identity (as was done for the sum $S_{1,3,1}(\C^d)$). Notice also that in most of the discarded parts in the sum $S_3(\C^d)$ we have $p_1 \leq X^{1/2-\sigma}$ and $p_1p_2 > X^{1/2+\sigma}$, which implies $p_2 > X^{2 \sigma}.$ 
\end{remark}

The Python 3.7 codes for computing the Buchstab integrals are available at (in the order of appearance)

\begin{tabular}{ c c }
$S_{1,3,1}$ & \url{http://codepad.org/QtJ3a0kq} \\
$S_{1,3,2}$, four dimensional prime part & \url{http://codepad.org/CBFtL1Tr} \\
$S_{1,3,2}$, five dimensional almost-prime part & \url{http://codepad.org/5FJpMXK6} \\
$S_{2}$, two dimensional prime part & \url{http://codepad.org/PF7WN4jj} \\
$S_{2}$, three dimensional almost-prime part & \url{http://codepad.org/EMFSgTzN} \\
$S_{3}$ & \url{http://codepad.org/lSHctNzv}
\end{tabular}

\section{Incomplete exponential sums} \label{expsumsection}
In this section we give bounds for incomplete exponential sums of type
\begin{align*}
S= \sum_{n} \psi_N(n) e_q(f(n)),
\end{align*}
where $q$ is a cube-free integer, $\psi_N$ is a smooth function with support of length $N,$ and $f=f_1/f_2$ with $f_1,f_2 \in \Z[X].$ These bounds are required for the proofs of Propositions \ref{typei}, \ref{typeii}, and \ref{typeipure}. The arguments are similar to those in \cite[Section 4]{polymath}, except that since we are considering distribution modulo squares, we need to deal with cube free moduli instead of squarefree. We complete the sum by using the P\'olya-Vinogradov method (Lemma \ref{expsum1lemma}) or Heath-Brown's $q$-van der Corput method (Lemma \ref{expsum2lemma}), which means that we require bounds for complete exponential sums
\begin{align*}
 \sum_{n \in \Z/p^m \Z} e_{p^m}(g(n)),
\end{align*}
where $m \in \{1,2\}.$ For $m=1$ we use the Weil bound similarly as in \cite[Section 4]{polymath} to show square root cancellation. For $m=2$ we use the work of Cochrane and Zheng \cite{cz} (cf. Lemma \ref{czlemma} below). Using  the P\'olya-Vinogradov method, for $N<q$ and $(q,f)=1$, we get  a bound $|S| \,\pprec q^{1/2}.$ If the modulus factorizes suitably $q=rs,$ then the $q$-van der Corput method gives a bound $|S| \, \pprec N^{1/2}(r +s^{1/2})^{1/2}.$ We can use smoothness of the modulus $q$ to optimize the factorization, which yields $|S| \pprec N^{1/2} q^{1/6} X^{\delta/6}.$ This is better than the P\'olya-Vinogradov bound if $N$ is a bit less than $q^{2/3}.$  For further discussion of the methods used, we refer to \cite[Section 4]{polymath}.

\subsection{Preliminaries}
 We first record some auxiliary results which can be found in \cite{polymath}:
\begin{lemma} (cf. \cite[Lemma 1.4]{polymath}) \label{gcdsum} Let $L \geq 1.$ For any integer $q \neq 0$ we have
\begin{align*}
\sum_{1 \leq \ell \leq L} (\ell, q) \leq \tau(q) L.
\end{align*}
\end{lemma}

\begin{lemma} \emph{\textbf{(Chinese Remainder Theorem).}} (cf. \cite[Lemma 4.4]{polymath}) \label{crt} Let $q_1,q_2, \dots, q_k$ be pairwise coprime positive integers and $q=q_1\cdots q_k.$ Then for any integers $a$ and $b$
\begin{align*}
e_q (a/b) = \prod_{j=1}^k e_{q_j} \bigg( \frac{a}{bq/q_j}\bigg),
\end{align*}
and for any rational function $f=f_1/f_2$ with $f_1,f_2 \in \Z[X]$ we have
\begin{align*}
\sum_{n \in \Z / q \Z} e_q (f(n)) = \prod_{j=1}^k \bigg(  \sum_{n \in \Z / q_j \Z} e_{q_j} \bigg( \frac{f(n)}{q/q_j} \bigg) \bigg).
\end{align*}
\end{lemma}

\begin{lemma} \label{complete}\emph{\textbf{(Completion of sums).}} (cf. \cite[Lemma 4.9(i), (4.14)]{polymath}) Let $M \geq 1$ and let $\psi_M$ be a function on $\R$ defined by
\begin{align*}
\psi_M(x)= \psi \bigg( \frac{x-x_0}{M} \bigg),
\end{align*}
where $\psi$ is a $C^\infty$-smooth function supported on some compact interval $[c,C]$ satisfying
\begin{align*}
\psi^{(j)}(x) \, \ll \log^{\mathcal{O}_j(1)} M
\end{align*}
for all $j \geq 0.$ Let
\begin{align*}
M' := \sum_m \psi_M(m) \, \ll M \log^{\mathcal{O}(1)}M.
\end{align*}
Then for any integer $q \geq 1$ and function $f: \Z / q \Z \to \mathbb{C}$ we have
\begin{align*}
\bigg | \sum_m  \psi_M(m) f(m) - \frac{M'}{q} \sum_{m \in \Z/q \Z} f(m) \bigg| \, \ll_{A,\epsilon} \, (\log^{\mathcal{O}(1)}M) \frac{M}{q}\sum_{0 < |\xi| \leq qM^{-1+\epsilon}} \bigg| \sum_{m \in \Z/q \Z} f(m) e_q(\xi m) \bigg | \\
+ M^{-A}\bigg| \sum_{m \in \Z/q \Z} f(m) \bigg|.
\end{align*}
\end{lemma}

\begin{lemma} \emph{\textbf{(Truncated Poisson summation formula).}} (cf. \cite[Lemma 4.9(ii)]{polymath}) \label{poisson}
 Let $M \geq 1$ and let $\psi_M$ be a function on $\R$ defined by
\begin{align*}
\psi_M(x)= \psi \bigg( \frac{x-x_0}{M} \bigg),
\end{align*}
where $\psi$ is a $C^\infty$-smooth function supported on some compact interval $[c,C]$ satisfying
\begin{align*}
\psi^{(j)}(x) \, \ll \log^{\mathcal{O}_j(1)} M
\end{align*}
for all $j \geq 0.$ Let
\begin{align*}
M' := \sum_m \psi_M(m) \, \ll M \log^{\mathcal{O}(1)}M.
\end{align*}
Let $I$ be a finite index set, $c_i$ complex numbers for $i \in I$, and $a_i \,(q)$ residue classes for $i \in I$. Then
\begin{align*}
\bigg| \sum_{i \in I} c_i \sum_{m} \psi_M(m) 1_{m = a_i \,\, (q)} - \frac{M'}{q}  \sum_{i \in I} c_i\bigg| \, \ll_{A,\epsilon} (\log^{\mathcal{O}(1)}M)\frac{M}{q} \sum_{0< |h| \leq q M^{-1+\epsilon}} &\bigg| \sum_{i\in I} c_i e_q(a_ih) \bigg| \\
 & \,+ \,  M^{-A} \sum_{i \in I} |c_i|.
\end{align*}
\end{lemma}

For any integer $q$ and polynomial $f_1(X) = \sum_{i=0}^d a_i X^i \in \Z[X]$, define 
\begin{align*}
(f_1,q):= \gcd(q,a_0,a_1,\cdots, a_d).
\end{align*}
If $f_2 \in \Z[X]$ is such that $(f_2,q)=1,$ then we set $(f_1/f_2,q):=(f_1,q).$

\begin{lemma} \textbf{\emph{(Weil bound).}} (cf. \cite[Proposition 4.6]{polymath}) \label{weil}
Let $p$ be a prime and  $f=f_1/f_2$ for coprime polynomials $f_1,f_2 \in \Z[X]$ such that  $(f_1,p)=(f_2,p)=(f',p)=1.$ Then
\begin{align*}
\bigg| \sum_{n \in \Z/p\Z} e_p(f(n)) \bigg| \, \ll p^{1/2},
\end{align*}
where the implicit constant depends only on the degrees of $f_1$ and $f_2$.
\end{lemma}

The next lemma contains an exponential sum estimate for complete sums modulo $p^2$ by Cochrane and Zheng \cite{cz}. For the lemma, let $p$ be a prime and $f=f_1/f_2$ for coprime polynomials $f_1,f_2  \in \Z[X]$ such that $(p,f_1)=(p,f_2)=(p,f')=1$. We say that $\alpha \, \, (p)$ is a critical point modulo $p$ if 
\begin{align*}
f'(\alpha) \equiv 0 \quad (p).
\end{align*}
For an exponential sum
\begin{align*}
S= \sum_{n \in \Z/p^2\Z} e_{p^2}(f(n)),
\end{align*}
write $S= \sum_{\alpha=0}^{p-1} S_\alpha$ where
\begin{align*}
S_\alpha :=  \sum_{\substack{n \in \Z/p^2\Z \\ n \equiv \alpha \,\, (p)}} e_{p^2}(f(n)).
\end{align*}
By \cite[Theorem 3.1]{cz} we have $S_\alpha=0$ if $\alpha$ is not a critical point. Using the trivial bound $|S_\alpha| \leq p$ for the critical points $\alpha$, we obtain
\begin{lemma} \emph{\textbf{(Cochrane-Zheng bound).}}\label{czlemma} Let $p$ be a prime. Let $f=f_1/f_2$ for coprime polynomials $ f_1,f_2 \in \Z[X],$ satisfying $(p,f_1)=(p,f_2)=(p,f')=1.$ Then
\begin{align*}
\bigg| \sum_{n \in \Z/p^2\Z} e_{p^2}(f(n)) \bigg| \, \leq  (\deg(f_1) + \deg(f_2)) p.
\end{align*}

\end{lemma}

\subsection{P\'olya-Vinogradov method}
Here we give the P\'olya-Vinogradov bound for short exponential sums. The statement and the proof are similar to the second bound in \cite[Corollary 4.16]{polymath}. In the first pass the reader may wish to consider the special case $b=1,$ $(d_1,d_2)=1$, and $(q,f)=1$.

\begin{lemma} \label{expsum1lemma} Let  $d_1$ and $d_2$ be cube free positive integers. Suppose that $b$ is a divisor of  $[d_1,d_2]$ with $(b,[d_1,d_2]/b)=1$. Let $c_1,c_2,$ and $\tau$ be integers, and define a rational function $f$ by
\begin{align*}
e_{d_1} \bigg ( \frac{c_1}{n} \bigg ) e_{d_2} \bigg ( \frac{c_2}{n+ \tau} \bigg) = e_{[d_1,d_2]} (f(n)).
\end{align*}
Denote
\begin{align*}
q := [d_1,d_2]/b, \quad \quad q_1 := \frac{q}{(q,f)}\quad \quad \delta_i:=\frac{d_i}{(d_1,d_2)}, \quad \quad \delta_i':=\frac{\delta_i}{(b,\delta_i)}.
\end{align*}
For $\delta_0:=(q,(d_1,d_2)),$ assume that $(q/\delta_0,\delta_0)=1.$ Let $t$ be any residue class modulo $b.$ Then
\begin{align*}
S:= \bigg| \sum_{n \equiv t \, (b)}  \psi_N(n) e_{[d_1,d_2]} (f(n)) \bigg| \, \pprec \,  \q_1^{1/2} + \frac{N}{b} \frac{(c_1,\delta_1')}{\delta_1'} \frac{(c_2,\delta_2')}{\delta_2'} .
\end{align*}
\end{lemma}
\begin{proof}
By a change of variables and by the Chinese Remainder Theorem (Lemma \ref{crt}), we have
\begin{align*}
S=  \bigg| \sum_{n}  \psi_N(b n + t) e_{q_1} (\tilde{f}(b n + t)) \bigg|,
\end{align*}
where $\tilde{f}(n):= b^{-1}f(n)/ (\q,f)$ (here the division of $f$ by $(\q,f)$ is computed in $\Z$ since it is possible that $(q_1,(q,f))>1$). This is an incomplete exponential sum of length $N/b$.  Thus, by Lemma \ref{complete} we get
\begin{align} \nonumber
S  & \ll_{\epsilon}   1+  \frac{N^{1+\epsilon}}{b q_1 }\sum_{0 < |\xi| \leq q_1 b N ^{-1+\epsilon}}  \bigg| \sum_{n \in \Z/q_1 \Z} e_{q_1} ( \tilde{f}(b n + t)+ \xi n) \bigg | +  \frac{N^{1+\epsilon}}{bq_1} \bigg | \sum_{n \in \Z/q_1 \Z} e_{q_1} (\tilde{f}(b n + t)) \bigg | \\  \label{polybound}
&=1+  \frac{N^{1+\epsilon}}{b q_1 }\sum_{0 < |\xi| \leq q_1 b N ^{-1+\epsilon}}   \bigg| \sum_{n \in \Z/q_1 \Z} e_{q_1} (\tilde{f}(n )+ \xi b^{-1} n) \bigg |  +  \frac{N^{1+\epsilon}}{bq} \bigg | \sum_{n \in \Z/q \Z} e_{q} (f(n)/b) \bigg | \\ \nonumber
&=: 1+S_1 +S_2
\end{align}
by a change of variables and Lemma \ref{crt}.

For the second sum we write $q = \delta_0 \delta_1' \delta_2'$, where $\delta_0=(q,(d_1,d_2)).$ Since $(q/\delta_0,\delta_0)=1$ by assumption, we get by the Chinese Remainder Theorem (Lemma \ref{crt}) and by the definition of $f$
\begin{align*}
S_2 \leq \frac{N^{1+\epsilon}}{bq} \delta_0 \bigg |  \sum_{\substack{n \in \Z/\delta_1' \Z }} e_{\delta_1'} \bigg(\frac{c_1 }{b  \delta_2'(d_1',d_2') n} \bigg) \bigg | \bigg | \sum_{n \in \Z/\delta_2' \Z} e_{\delta_2'} \bigg ( \frac{c_2}{b \delta_1' (d_1',d_2') (n +\tau)} \bigg)\bigg |
\end{align*}
After a change of variables we get by a standard bound for Ramanujan's sums
\begin{align*}
S_2& \leq \frac{N^{1+\epsilon}}{bq} \delta_0 \bigg | \sum_{\substack{n \in \Z/\delta_1' \Z \\ (n,\delta_1')=1}}e_{\delta_1'} ( c_1 n ) \bigg | \bigg | \sum_{\substack{n \in \Z/\delta_2' \Z \\ (n,\delta_2')=1}} e_{\delta_2'} (c_2 n ) \bigg | \\
& \leq   \frac{N^{1+\epsilon}}{bq} \delta_0 (c_1,\delta_1')(c_2,\delta_2') = \frac{N^{1+\epsilon}}{b} \frac{(c_1,\delta_1')}{\delta_1'} \frac{(c_2,\delta_2')}{\delta_2'}.
\end{align*}

To bound the first sum in (\ref{polybound}), for $p|q_1$ let $p'$ denote $p$ or $p^2$ so that $\prod_{p | \q_1} p' = \q_1.$ Then by Lemma \ref{crt}
\begin{align*}
 \bigg| & \sum_{n \in \Z/q_1  \Z} e_{\q_1} ( \tilde{f}(n )+ \xi b^{-1} n) \bigg | 
= \prod_{p | \q_1 } \bigg| \sum_{n \in \Z/p'\Z } e_{p'} \bigg( \frac{\tilde{f}(n)+ \xi b^{-1} n}{\q_1/p'}\bigg)  \bigg| 
\end{align*}
 Note that for $g(n):=  \tilde{f}(n )+ \xi b^{-1} n,$ for any prime $p| \q_1$ we have $(p,g) \leq (p,\xi)$ and  $(p,g') \leq (p,\xi),$ since if we write $g=g_1/g_2$ for coprime polynomials $g_1,g_2,$ then the leading coefficient of $g_1$ is $\xi b^{-1}$.  Thus, by using trivial bounds for primes $p| \xi$, and  Lemma \ref{weil} (if $p'=p$) and  Lemma \ref{czlemma} (if $p'=p^2$) for $p \nmid \xi$, we get 
\begin{align*}
 \bigg| \sum_{n \in \Z/q_1  \Z} e_{\q_1} ( \tilde{f}(n )+ \xi b^{-1} n) \bigg | & \pprec  \prod_{p | \q_1} (\xi,p) (p')^{1/2}   \, \leq    (\xi,\q_1)  \q_1^{1/2} .
\end{align*}
Using Lemma \ref{gcdsum} we obtain
\begin{align*}
 S_1 \pprec \frac{N^{1+\epsilon}}{b \q_1}\sum_{0 < |\xi| \leq  \q_1 b N^{-1+\epsilon}}  (\xi,\q_1) \q_1^{1/2} \, \pprec \, \,  \q_1^{1/2}.
\end{align*}
\end{proof}

\subsection{Heath-Brown's $q$-van der Corput method}
In this section we apply the $q$-van der Corput method to obtain a bound for exponential sums, which for short lengths performs better than the P\'olya-Vinogradov bound. Compared to Lemma \ref{expsum1lemma}, we require the extra assumption that the modulus is smooth, so that we can obtain a suitable factorization. The statement and the proof are similar to the first bound in \cite[Proposition 4.16]{polymath}.
\begin{lemma} \label{expsum2lemma} Let  $d_1$ and $d_2$ be cube free positive integers with no prime factors $\gg X^{\delta/2}$. Suppose that $b$ is a divisor of  $[d_1,d_2]$ with $(b,[d_1,d_2]/b)=1$. Let  $c_1,c_2,$ and $\tau$ be integers, and define a rational function $f$ by
\begin{align*}
e_{d_1} \bigg ( \frac{c_1}{n} \bigg ) e_{d_2} \bigg ( \frac{c_2}{n+ \tau} \bigg) = e_{[d_1,d_2]} (f(n)).
\end{align*}
Denote
\begin{align*}
q := [d_1,d_2]/b, \quad \quad q_1 := \frac{q}{(q,f)}\quad \quad \delta_i:=\frac{d_i}{(d_1,d_2)}, \quad \quad \delta_i':=\frac{\delta_i}{(b,\delta_i)}.
\end{align*}
For $\delta_0:=(q,(d_1,d_2)),$ assume that $(q/\delta_0,\delta_0)=1.$ Let $t$ be any residue class modulo $b.$ Then
\begin{align*}
 S:=  \bigg| \sum_{n \equiv t \, (b)}  \psi_N(n) e_{[d_1,d_2]} (f(n)) \bigg| \, \pprec   \frac{N^{1/2}}{b^{1/2}} \q_1^{1/6}X^{\delta/6} + \frac{N}{b} \frac{(c_1,\delta_1')}{\delta_1'} \frac{(c_2,\delta_2')}{\delta_2'} .
\end{align*}
\end{lemma}
\begin{proof}
Since $\q_1$ is cube free and has no prime factors $\gg X^{\delta/2},$ we may factorize $q_1=r s$ with $(r,s)=1,$ where
 \begin{align*}
 X^{-2\delta/3} q_1^{1/3} \ll r \ll X^{\delta/3} q_1^{1/3} \quad \text{and} \quad  X^{-\delta/3} q_1^{2/3} \ll s \ll X^{2\delta/3} q_1^{2/3}.
 \end{align*}
  We may assume that $N/b < s,$ since otherwise by  Lemma \ref{expsum1lemma}
\begin{align*}
S  \, &\pprec (rs)^{1/2} + \frac{N}{b} \frac{(c_1,\delta_1')}{\delta_1'} \frac{(c_2,\delta_2')}{\delta_2'}  \ll \frac{N^{1/2}}{b^{1/2}}\q_1^{1/6}X^{\delta/6} + \frac{N}{b} \frac{(c_1,\delta_1')}{\delta_1'} \frac{(c_2,\delta_2')}{\delta_2'}.
\end{align*}
  
  We may also assume that $N/b \geq r,$ because in the opposite case we get by a trivial bound
  \begin{align*}
  S  \pprec   \frac{N}{b}  < \frac{N^{1/2}}{b^{1/2}} r^{1/2} \leq  \frac{N^{1/2}}{b^{1/2}} q_1^{1/6} X^{\delta/6},
  \end{align*}
  which is sufficient. Hence, we may define $K := \lfloor N/ br \rfloor \geq 1.$ We then obtain   for $\tilde{f}(n):= b^{-1}f(n) /(\q,f) $ (with the division of $f$ by $(\q,f)$  computed in $\Z$ since it is possible that $(q_1,(q,f))>1$) by the Chinese Remainder Theorem (Lemma \ref{crt}) and by a change of variables
  \begin{align*}
S  = \bigg| \sum_{n } \psi_N(bn+t) e_{q_1}(\tilde{f}(bn+t)) \bigg | =  \bigg| \frac{1}{K}\sum_{n} \sum_{k=1}^K \psi_N(bn+t+kr) e_{q_1}(\tilde{f}(bn+t+kr)) \bigg | .
  \end{align*}
By Lemma \ref{crt} and periodicity
  \begin{align*}
   e_{q_1}(\tilde{f}(bn+t+kr)) &=  e_{r}( \tilde{f}(bn+t+kr)/\s ) e_{s}(\tilde{f}(bn+t+kr)/\r )  \\
   & = e_{r}( \tilde{f}(bn+t)/\s) e_{s}(\tilde{f}(bn+t+kr)/\r) .
  \end{align*}
  Hence, denoting $g(n):= r^{-1}\tilde{f}(n)$,  we obtain by Cauchy-Schwarz
  \begin{align*}
  S & \leq  \sum_{n} \bigg|  \frac{1}{K}  \sum_{k=1}^K \psi_N(bn+t+kr) e_{s}(g(bn+t+kr)) \bigg |  \\
  & \ll \frac{N^{1/2}}{b^{1/2}}\bigg ( \sum_{n} \bigg|  \frac{1}{K} \sum_{k=1}^K \psi_N(bn+t+kr) e_{s}(g(bn+t+kr)) \bigg |^2 \bigg)^{1/2} \\
  &= N^{1/2}  \bigg (   \frac{1}{K^2} \sum_{k_1, k_2=1}^K S(k_1,k_2) \bigg)^{1/2} ,
  \end{align*}
  where
  \begin{align*}
  S(k_1,k_2):=   \sum_{n}  \psi_N(bn+t+k_1 r) \psi_N(bn+t+k_2r) e_{s}(g(bn+t+k_1r)-g(bn+t+k_2r)).
  \end{align*}
By a trivial bound we have
  \begin{align*} 
  \frac{1}{K^2}\sum_{k=1}^K  |S(k,k)|\, \ll N/bK \ll \, r .
\end{align*}
For $k_1\neq k_2$ we complete the sum by using Lemma \ref{complete} and use a change of variables to obtain
 \begin{align*} 
 |  S(k_1,k_2) | \,  \ll_{\epsilon} & \,\, 1+  \frac{N^{2\epsilon}}{T}\sum_{0 \leq |\xi| \leq T}  \bigg| \sum_{n \in \Z/s \Z} e_{s} (g(n+k_1r)-g(n+k_2r) + \xi b^{-1} n )\bigg |,
 \end{align*}
 where $T:= s b N^{-1+\epsilon} \geq 1.$

For $p|s$, let $p'$ denote $p$ or $p^2$ so that $\prod_{p | \s} p' = \s$. For each $\xi$ we use the Chinese Remainder Theorem (Lemma \ref{crt}) to get 
\begin{align*}
\bigg| \sum_{n \in \Z/s \Z} e_{s}(g(n+k_1r) & -g(n+k_2r) + \xi b^{-1} n) \bigg | \\
& = \prod_{\substack{p| s}}\bigg| \sum_{n \in \Z/ p' \Z} e_{p'} \bigg(  \frac{g(n +k_1r)-g(n+k_2r) + \xi b^{-1} n}{s/p'} \bigg) \bigg|.
\end{align*} 

Denote $h(n):=  g(n +k_1r)-g(n+k_2r) + \xi b^{-1} n.$ We claim that $(p,h) \leq (p,k_1-k_2)$ and $(p,h') \leq (p,k_1-k_2)$ for any $p|s$ with $p \gg 1$.  To see this, we note that if $(p,\xi)=1,$ then $(p,h)=(p,h')=1$ since $\xi$ is the leading coefficient of the numerator for both $h$ and $h'$. If $p| \xi,$ then by a change of variables $(p,h)=(p,\tilde{h})$ and $(p,h')=(p,\tilde{h}')$, where $\tilde{h}(n)=g(n +(k_1-k_2)r)-g(n).$ Hence, we need to show that for any integer $\ell$ we have $p|(g(X+\ell)-g(X))$ only if $p|\ell$ (the argument that follows is essentially the same as in the proof of \cite[Proposition 4.12]{polymath}). To see this, suppose for the sake of contradiction that $p|(g(X+\ell)-g(X))$ but $p\nmid \ell$. Then by induction $p|(g(X+i\ell)-g(X))$ for every $i \in \N$. But since $p \nmid \ell,$ this implies by periodicity that the value of $g$ modulo $p$ is constant. Since $(p,g)=1$ and $g$ is by definition of the form $g_1/g_2$ for $g_2(n)=n(n+\tau)$ and for some $g_1 \in \Z[X]$ with $\deg g_1 \leq 1$, this is a contradiction if $p$ is sufficiently large.

Hence, by using the trivial bound for $p | (k_1-k_2),$ and Lemma \ref{weil} (if $p'=p$) and  Lemma \ref{czlemma} (if $p'=p^2$) for $p \nmid (k_1-k_2)$, we get
\begin{align*}
&\bigg| \sum_{n \in \Z/s \Z} e_{s}(g(n+k_1r)-g(n+k_2r) + \xi b^{-1} n) \bigg | \pprec \, \prod_{\substack{p| s}} (k_1-k_2,p) (p')^{1/2} \, \leq (k_1-k_2, s)s^{1/2}.
\end{align*}
Since $T \geq 1,$ we get by Lemma \ref{gcdsum}
\begin{align*}
S \, \pprec \frac{N^{1/2}}{b^{1/2}} \bigg( r + \frac{1}{K^2} \sum_{\substack{k_1,k_2 \leq K \\ k_1\neq k_2}} (k_1-k_2, s) s^{1/2} \bigg)^{1/2} \, \pprec \frac{N^{1/2}}{b^{1/2}} \q_1^{1/6} X^{\delta/6}.
\end{align*}
\end{proof}

\section{Type I/II estimate} \label{typeisection}
In this section we prove Proposition \ref{typei} by using Poisson summation formula and Cauchy-Schwarz. First we apply  finer-than-dyadic decomposition to replace $\psi(\ell mn/X)$ by $\psi(\ell \tilde{M} \tilde{N}/X )$ for some $\tilde{M} \sim M, \tilde{N} \sim N$: Let $\Delta:=X^{-\eta/4}$ for some small $\eta>0$, and let $\tilde{M}$ and $\tilde{N}$ run over numbers of the form $(1+\Delta)^j$ for $j \in \N$. Then
\begin{align*}
&\Sigma_d :=  \bigg| \sum_{\substack{\ell mn \equiv a \, (d^2) \\ m \sim M, \, \, n \sim N}} \alpha(m) \beta(n) \psi(\ell mn/X) - \frac{1}{\phi(d^2)} \sum_{\substack{( \ell mn,d^2)=1 \\ m \sim M, \, \, n \sim N}} \alpha(m) \beta(n) \psi(\ell mn/X)  \bigg|  \\
&= \bigg|\sum_{\substack{\tilde{M} \sim M \\ \tilde{N} \sim N }} \bigg(\sum_{\substack{\ell mn \equiv a \, (d^2) \\ m \in (\tilde{M}, \tilde{M}(1+\Delta)] \\ n \in (\tilde{N}, \tilde{N}(1+\Delta)]}} \alpha(m) \beta(n) \psi(\ell mn/X) - \frac{1}{\phi(d^2)} \sum_{\substack{( \ell mn,d^2)=1 \\m \in (\tilde{M}, \tilde{M}(1+\Delta)] \\ n \in (\tilde{N}, \tilde{N}(1+\Delta)] }} \alpha(m) \beta(n) \psi(\ell mn/X) \bigg)  \bigg|.
\end{align*}
By the mean value theorem for all $m \in (\tilde{M}, \tilde{M}(1+\Delta)]$ and $ n \in (\tilde{N}, \tilde{N}(1+\Delta)]$ we have
\begin{align*}
\psi(\ell m n /X) = \psi (\ell \tilde{M} \tilde{N}/X) + \mathcal{O}(X^{-\eta/4}).
\end{align*}
 Hence, by the triangle inequality
\begin{align*}
\Sigma_d  \ll &\bigg|\sum_{\substack{\tilde{M} \sim M \\ \tilde{N} \sim N }} \bigg(\sum_{\substack{\ell mn \equiv a \, (d^2) \\ m \in (\tilde{M}, \tilde{M}(1+\Delta)] \\ n \in (\tilde{N}, \tilde{N}(1+\Delta)]}} \alpha(m) \beta(n) \psi(\ell \tilde{M} \tilde{N}/X) - \frac{1}{\phi(d^2)} \sum_{\substack{( \ell mn,d^2)=1 \\m \in (\tilde{M}, \tilde{M}(1+\Delta)] \\ n \in (\tilde{N}, \tilde{N}(1+\Delta)] }} \alpha(m) \beta(n) \psi(\ell \tilde{M} \tilde{N}/X) \bigg)  \bigg| \\
& \hspace{50pt}+ X^{-\eta/4}\sum_{\substack{\ell m n \asymp X \\ \ell m n \equiv a \, (d^2)}} |\alpha(m)\beta(n)|\,\, + \, \, X^{-\eta/4} \frac{1}{\phi(d^2)}\sum_{\substack{\ell m n \asymp X}} |\alpha(m)\beta(n)| \\
 \pprec &  \, X^{\eta/2}\max_{\substack{\tilde{M} \sim M \\ \tilde{N} \sim N }} \Sigma_d(\tilde{M}, \tilde{N})  \, \,+ \,\, X^{1-\eta/4}/D,
\end{align*}
where (after absorbing the restrictions $m \in (\tilde{M}, \tilde{M}(1+\Delta)]$ and $n \in (\tilde{N}, \tilde{N}(1+\Delta)]$ respectively into the coefficients $\alpha(m)$ and $\beta(n)$)
\begin{align*}
\Sigma_d(\tilde{M}, \tilde{N}) :=  \bigg| \sum_{\substack{\ell mn \equiv a \, (d^2) \\ m \sim M, \, \, n \sim N}} \alpha(m) \beta(n) \psi(\ell \tilde{M} \tilde{N} /X) - \frac{1}{\phi(d^2)} \sum_{\substack{( \ell mn,d^2)=1 \\ m \sim M, \, \, n \sim N}} \alpha(m) \beta(n) \psi(\ell \tilde{M} \tilde{N}/X)  \bigg|. 
\end{align*}
Hence, it suffices to show that for any $\tilde{M} \sim M $ and $\tilde{N} \sim N$ we have
\begin{align*}
\Sigma_d(\tilde{M}, \tilde{N}) \ll X^{1-\eta}/D.
\end{align*}

  Write
\begin{align*}
 \Sigma_d (\tilde{M},\tilde{N}) :=  \bigg| \sum_{\substack{\ell mn \equiv a \, (d^2) \\ m \sim M, \, \, n \sim N}} \alpha(m) \beta(n) \psi(\ell \tilde{M} \tilde{N}/X) - \frac{1}{\phi(d^2)} \sum_{\substack{( \ell mn,d^2)=1 \\ m \sim M, \, \, n \sim N}} \alpha(m) \beta(n) \psi(\ell \tilde{M} \tilde{N}/X)  \bigg|  \, \\
\leq \frac{1}{\phi(d^2)}\sum_{(b,d^2)=1} \bigg| \sum_{\substack{\ell mn \equiv a \, (d^2) \\ m \sim M, \, \, n \sim N}} \alpha(m) \beta(n) \psi(\ell \tilde{M} \tilde{N}/X) -  \sum_{\substack{\ell mn \equiv b \, (d^2) \\ m \sim M, \, \, n \sim N}} \alpha(m) \beta(n) \psi( \ell \tilde{M} \tilde{N}/X)  \bigg|
\end{align*}
Applying the Poisson summation formula (Lemma \ref{poisson}) we get
\begin{align*}
\Sigma_d(\tilde{M},\tilde{N}) \, \ll_{\epsilon} \, 1 + X^{2\epsilon} \max_{(b,d^2)=1} \widehat{\Sigma}(b) ,
\end{align*}
where for $H:= X^\epsilon D/L $ we have
\begin{align*}
\widehat{\Sigma} (b) := \frac{1}{H} \sum_{0 < |h| \leq H} \bigg | \sum_{\substack{ (mn,d^2)=1 \\ m \sim M, \, \, n \sim N }} \alpha(m) \beta(n) e_{d^2}\bigg(\frac{bh}{mn}\bigg) \bigg |.
\end{align*}
To remove the coefficients $\alpha(m)1_{(m,d^2)=1}$ we use Cauchy-Schwarz to get (for some phases $c_h \in \mathbb{C}$ and for $\psi_M(m)$ a $C^\infty$-smooth majorant of $1_{m \sim M}$)
\begin{align} \label{cs1}
\widehat{\Sigma} (b) &=   \sum_{\substack{ (m,d^2)=1 \\ m \sim M }} \alpha(m) \frac{1}{H} \sum_{0 < |h| \leq H} c_h \sum_{\substack{ (n,d^2)=1 \\   n \sim N }}\beta(n) e_{d^2}\bigg(\frac{bh}{mn}\bigg)  \\ \nonumber 
&\pprec M^{1/2} \bigg ( \sum_{m} \psi_M(m) \frac{1}{H^2} \bigg |  \sum_{0 < |h| \leq H} c_h \sum_{\substack{ (n,d^2)=1 \\   n \sim N }}\beta(n) e_{d^2}\bigg(\frac{bh}{mn}\bigg)  \bigg |^2 \bigg )^{1/2} \\ \nonumber
& =  M^{1/2}  \bigg ( \frac{1}{H^2}   \sum_{0 < |h_1 |, |h_2| \leq H} \sum_{\substack{ (n_1n_2,d^2)=1 \\   n_1, n_2 \sim N }} c_{h_1}\overline{c_{h_2}} \beta(n_1)\overline{\beta(n_2)} \sum_{m} \psi_M(m)  e_{d^2}(\gamma / m)  \bigg )^{1/2} .
\end{align}
where $\gamma := b(h_1/n_1 - h_2/n_2)$. By applying Lemma \ref{expsum1lemma} with $d_1=d^2$ and $d_2=1,$ we obtain 
\begin{align*}
 \bigg | \sum_{m} \psi_M(m) e_{d^2}(\gamma /m) \bigg |    \pprec d + \frac{M(h_1n_2-h_2n_1,d^2)}{d^2} .
\end{align*}
Hence, summing over $c:= (h_1n_2-h_2n_1,d^2)$ we have
\begin{align*}
\widehat{\Sigma} (b) \, & \pprec M^{1/2} \bigg ( \sum_{c | d^2} \frac{1}{H^2} \sum_{\substack{h_1,h_2,n_1,n_2 \\ h_1n_2 - h_2n_1 \equiv 0 \, (c)}} \bigg( D^{1/2}+ \frac{M c}{D} \bigg)\bigg)^{1/2} \\
& \pprec M^{1/2} \bigg(  \sum_{c | d^2}\bigg(  D^{1/2} + \frac{M c}{D}  \bigg) \frac{1}{H^2}\sum_{\ell_1 \ll HN} \sum_{\substack{\ell_2 \, \ll HN \\ \ell_2 \equiv \ell_1  \, (c)} } 1  \bigg)^{1/2} \\
& \ll M^{1/2} \bigg(  \sum_{c | d^2}\bigg( D^{1/2} + \frac{M c}{D}  \bigg)  \bigg(\frac{N^2}{c} + \frac{N}{H} \bigg) \bigg)^{1/2} \\
& \pprec  M^{1/2} \bigg(N^2 D^{1/2}  +  \frac{D^{1/2}N}{H} +\frac{MN^2}{D} +  \frac{MN}{H} \bigg)^{1/2} \, \\
&\ll M^{1/2} N D^{1/4} + \frac{(LMN)^{1/2}}{D^{1/4}}+ \frac{M N}{D^{1/2}} + \frac{ M^{1/2}(LMN)^{1/2}}{D^{1/2}} \\
&\ll M^{1/2} N D^{1/4} + \frac{X^{1/2}}{D^{1/4}}+ \frac{M N}{D^{1/2}} +  \frac{ M^{1/2}X^{1/2}}{D^{1/2}} \, \ll \frac{X^{1-\eta}}{D},
\end{align*}
where we have used the fact that $D=X^{1/2+2\varpi},$ $LMN=X$, and $H=X^\epsilon D/L.$ On the last line the first term is sufficiently small since $N \leq X^{1/8 +\sigma/2 -5\varpi /2 - \eta},$ and the fourth term is sufficiently small since $M \leq X^{1/2-\sigma}$ with $\sigma/2 > 3 \varpi + \eta.$
\qed

\section{Type I estimate} \label{typeipuresection}
\subsection{Optimizing Cauchy-Schwarz}
Here we offer a heuristic explanation of the arguments that follow (cf. \cite[Section 15.3.1]{michel} for a similar heuristic). Suppose we want to bound a trilinear sum of the form
\begin{align*}
\sum_{m \sim M} \sum_{n \sim N} \sum_{d \sim D} \alpha_m \beta_n \gamma_d \Phi (m,n,d)
\end{align*}
by using Cauchy-Schwarz to replace $\alpha_m$ by a smooth function $\psi_M(m)$. There are two options how to do this, either
\begin{align*}
\bigg( \sum_{m \sim M} |\alpha_m|^2 \bigg)^{1/2} \bigg( \sum_{d \sim D} |\gamma_d|^2 \bigg)^{1/2}  \bigg( \sum_{m } \psi_M(m) \sum_{d\sim D} \bigg| \sum_{n \sim N} \beta_n \Phi (m,n,d) \bigg|^2 \bigg)^{1/2}
\end{align*}
or
\begin{align*}
\bigg( \sum_{m \sim M} |\alpha_m|^2 \bigg)^{1/2}  \bigg( \sum_{m } \psi_M(m)\bigg| \sum_{n \sim N}  \sum_{d\sim D}  \beta_n \gamma_d \Phi (m,n,d) \bigg|^2 \bigg)^{1/2}.
\end{align*}
We then have to control either 
\begin{align*}
 \sum_{n_1, n_2 \sim N} \beta_{n_1} \overline{\beta_{n_2}}  \sum_{d\sim D}  \sum_{m } \psi_M(m) \Phi (m,n_1,d)\overline{\Phi (m,n_2,d)}.
\end{align*}
or
\begin{align*}
 \sum_{n_1, n_2 \sim N} \beta_{n_1} \overline{\beta_{n_2}}  \sum_{d_1, d_2 \sim D}   \gamma_{d_1} \overline{ \gamma_{d_2}} \sum_{m } \psi_M(m)    \Phi (m,n_1,d_1)\overline{\Phi (m,n_2,d_2)}.
\end{align*}
In the off-diagonal case ($n_1 \neq n_2$ in the first sum,  $(n_1,d_1) \neq (n_2,d_2)$ in the second sum) we expect to be able to show cancellation in the sum over $m$. In the diagonal case we do not get any cancellation, but we hope that the diagonal is a small subset of the set of a variables. We are then faced with a trade-off:
\vspace{5pt} 

\emph{In the first case the ratio $1/N$  of the diagonal to the variable set is larger but the coefficient $\Phi (m,n_1,d)\overline{\Phi (m,n_2,d)}$ is simpler.}
\vspace{5pt} 

\emph{In the second case the ratio  $1/(ND)$  of the diagonal to the variable set is smaller but the coefficient $\Phi (m,n_1,d_1)\overline{\Phi (m,n_2,d_2)}$ is more complicated.}

We have already seen this in the proof of the Type I/II estimate, where in (\ref{cs1}) it was important to keep the sum over $h$ inside to make the diagonal contribution sufficiently small (however, there this did not cause any complications to the sum over $m$ in the off-diagonal case; in the Type I and Type II estimates we will not be so lucky). 

For the proofs of the Type I and Type II estimate we will make use of the fact that $d$ is well-factorable, so that we can split the sum over $d^2$ as
\begin{align*}
\sum_{d^2 \in \DD} = \sum_{r^2 \in \RR} \sum_{q^2 \in \QQ},
\end{align*}
and find a middle ground of the two alternatives by keeping the sum over $r$ outside and sum over $q$ inside; this idea goes back to the work of Fouvry and Iwaniec on equidistribution estimates with well-factorable weights \cite{fi}. The idea of using smooth moduli is due to Zhang \cite{zhang}. In the proof of the Type II estimate we find that later in the argument we need to split some sums a second time before another application of the Cauchy-Schwarz inequality. The factorization is always determined in such a way that the diagonal contribution is just small enough, so that the resulting sum over the smoothed variable is as simple as possible. Note that here the variables $d^2,r^2,q^2$ run over a sparse set. This has to be taken into account when deciding the factorization, which is the main reason why our Type II range is more difficult to handle than in \cite[Theorem 5.1]{polymath}.

\subsection{Proof of the Type I estimate}

In this section we prove Proposition \ref{typeipure}. The proof will already feature many of the ingredients that go into the proof of the Type II estimate in the next section, although here it is not necessary to fully optimize the argument. 

Similarly as in the proof of the Type I/II estimate, we can use finer-than-dyadic decomposition to replace $\psi(mn/X)$ by $\psi(\tilde{M} n/x)$ for some $\tilde{M} \sim M.$ Absorbing the condition $1_{m \sim M}$ into the coefficient $\alpha(m)$, by the Chinese Remainder Theorem and by triangle inequality we have
\begin{align*}
 \Sigma  :=\sum_{d^2 \in \DD} & \bigg | \sum_{\substack{mn \equiv a \, (d^2) }} \alpha(m)  \psi(\tilde{M}n/X) - \frac{1}{\phi(d^2)} \sum_{\substack{(mn,d^2)=1 }} \alpha(m)  \psi(\tilde{M}n/X) \bigg |  \\
& \leq \frac{1}{\phi(P_I^2)} \sum_{(b, P_I^2) =1} \sum_{d^2 \in \DD}  \bigg | \sum_{\substack{mn \equiv a \, (d^2) }} \alpha(m)  \psi(\tilde{M}n/X) -  \sum_{mn \equiv b \, (d^2) } \alpha(m)  \psi(\tilde{M}n/X) \bigg | 
\end{align*}
where $P_I= \prod_{p \in I} p$ for $I = \bigcup_{j=1}^K I_j.$ By the Poisson summation formula (Lemma \ref{poisson}) we obtain
\begin{align*}
 \Sigma    \ll_\epsilon  1 + \max_{(b, P_I^2) =1}  X^{2\epsilon} \hat{\Sigma}(b),
\end{align*}
where for $H := X^\epsilon D/N$
\begin{align*}
\hat{\Sigma}(b):=  \frac{1}{H} \sum_{1 \leq |h| \leq H} \sum_{d^2 \in \DD} \bigg| \sum_m \alpha(m) e_{d^2} (b h/m) \bigg|
\end{align*}
We now plan to use Cauchy-Schwarz to replace $\alpha (m)$ by a smooth function. In order to do this we need to split the sum over $d^2$ as follows: recall that by our set-up in the beginning of Section \ref{harmansection} we have $d^2=p_1^2\cdots p_K^2$ for primes $p_j \asymp P^{1/2} = D^{1/2K}$. Choose $K_0 \leq K$ so that for $R:=P^{K_0},$ $Q:=P^{K-K_0},$ $RQ=D$ we have
\begin{align*}
Q \in [X^{16 \varpi + 8 \delta}, X^{16\varpi + 9 \delta}].
\end{align*}
Define
\begin{align*}
\RR := \{p_1^2 \cdots p_{K_0}^2: \, p_j \in I_j \} \quad \text{and} \quad \QQ := \{p_{K_0+1}^2 \cdots p_{K}^2: \, p_j \in I_j \} 
\end{align*}
so that for $r^2 \in \RR$ and $q^2 \in \QQ$ we have $r^2 \asymp R$ and $q^2 \asymp Q$. Note that then $(r,q)=1.$ This factorization is determined so that we can control the diagonal contribution. For some $r^2 \in \RR$ we have by Cauchy-Schwarz
\begin{align*}
\hat{\Sigma} (b)& \ll R^{1/2}  \bigg| \sum_m \alpha(m) \frac{1}{H}\sum_{1 \leq |h| \leq H}  \sum_{q^2 \in \QQ} c_{r,q,h}  e_{r^2q^2} (bh/m) \bigg|  \\
& \pprec R^{1/2} M^{1/2} \bigg(  \sum_{m} \psi_M(m) \bigg|\frac{1}{H}\sum_{1 \leq |h| \leq H}\sum_{q^2 \in \QQ}  c_{r,q,h}  e_{r^2q^2} (bh/m)\bigg|^2 \bigg)^{1/2} \\
& \leq  R^{1/2} M^{1/2} \bigg( \frac{1}{H^2}\sum_{1 \leq |h_1|, |h_2| \leq H} \sum_{q_1^2, q_2^2 \in \QQ} \bigg|\sum_{m} \psi_M(m) e_{r^2 q_1^2} (bh_1/m) e_{r^2 q_2^2} (-bh_2/m) \bigg| \bigg)^{1/2},
\end{align*}
where $\psi_M$ is a $C^\infty$-smooth majorant for $1_{m \sim M}$. By the Chinese Remainder Theorem (Lemma \ref{crt}) we have for some integer $c=c(h_1,h_2,q_1,q_2,r)$
\begin{align*}
e_{r^2 q_1^2} (bh_1/m) e_{r^2 q_2^2} (-bh_2/m)  = e_{r^2 [q_1^2,q_2^2]} (bc/m)
\end{align*}
Applying Lemma \ref{expsum1lemma}  with $d_1=r^2[q_1^2,q_2^2] $ and $d_2=1$ we obtain (since $(b,r^2[q_1^2,q_2^2])=1$)
\begin{align*}
\bigg | \sum_{m} \psi_M(m) e_{r^2 [q_1^2,q_2^2]} (bc/m) \bigg|  \pprec r  q_1 q_2 + \frac{M}{r^2[q_1^2,q_2^2]} (c, r^2[q_1^2,q_2^2]).
\end{align*}
Expanding the definition of $c$ by using the Chinese Remainder Theorem we find (since $(r^2,q_1^2q_2^2)=1$)
\begin{align*}
\frac{M}{r^2[q_1^2,q_2^2]} (c, r^2[q_1^2,q_2^2]) \leq \frac{M}{r^2} (c, r^2) = \frac{M}{r^2} (h_1 q_2^2-h_2 q_1^2 , r^2).
\end{align*}
 Hence, we have
\begin{align*}
 \frac{1}{H^2}\sum_{1 \leq |h_1|, |h_2| \leq H} \sum_{q_1^2, q_2^2 \in \QQ} \bigg|\sum_{m} \psi_M(m) & e_{r^2 q_1^2} (bh_1/m) e_{r^2 q_2^2} (-bh_2/m)  \bigg| \\
 & \pprec R^{1/2} Q^2 +   \frac{M}{R}  \frac{1}{H^2}\sum_{1 \leq |h_1|, |h_2| \leq H} \sum_{q_1^2, q_2^2 \in \QQ} (h_1 q_2^2- h_2 q_1^2, r^2).
\end{align*}
Writing $\Delta = h_1 q_2^2- h_2 q_1^2,$ we have by Lemma \ref{gcdsum}
\begin{align*}
  \frac{1}{H^2}\sum_{1 \leq |h_1|, |h_2| \leq H} \sum_{q_1^2, q_2^2 \in \QQ} (h_1 q_2^2- h_2 q_1^2, r^2) & \pprec \frac{1}{H^2} \sum_{0 \leq |\Delta| \ll HQ} (\Delta, r^2) \sum_{1 \leq |h_1| \leq H} \sum_{q_2^2 \in \QQ} 1 \\
  & \pprec   R Q^{1/2}/H +  Q^{3/2} \leq R Q^{1/2}.
\end{align*}
Thus,
\begin{align*}
\hat{\Sigma}(b) &\pprec  R^{1/2} M^{1/2} ( R^{1/2} Q^2  + M Q^{1/2})^{1/2} \leq R^{3/4} Q M^{1/2} + R^{1/2} Q^{1/4} M \\
& = \frac{X}{\sqrt{RQ}} \bigg( \frac{R^{5/4} Q^{3/2}}{N M^{1/2}} + \frac{R Q^{3/4}}{ N}\bigg) \leq  \frac{X^{1-\eta}}{\sqrt{D}},
\end{align*}
since
\begin{align*}
\frac{R Q^{3/4}}{ N} = \frac{D}{N Q^{1/4}} \leq \frac{X^{1/2+2 \varpi}}{X^{1/2-2 \varpi - \delta} X^{4 \varpi + 2 \delta}} = X^{-\delta},
\end{align*}
and
\begin{align*}
\frac{R^5 Q^6}{N^4 M^2} = \frac{D^5 Q}{X^2 N^2} \leq \frac{X^{5/2 + 10 \varpi + 16 \varpi + 9 \delta}}{X^{3 - 4 \varpi - 2 \delta}} \leq X^{-\delta}
\end{align*}
by using
\begin{align*}
30 \varpi + 11 \delta \leq 1/2 - \delta.
\end{align*}
\qed

\section{Type II estimate} \label{typeiisection}
\subsection{Large sieve for sparse sets of moduli}
For preliminary reductions in the Type II estimate we require the large sieve inequality for sparse sets of moduli of Baier and Zhao \cite[Lemma 9]{bzsparse}. To state the lemma we need to define the notion of a well-distributed set (as in \cite[Sections 2]{bzsparse}): for any set $S$ of natural numbers define
\begin{align*}
S_t = \{q \in \N: qt\in S\}, \quad S_t(Q):= S_t \cap ( Q,2Q], \quad \text{and} \quad S(Q)= S_1(Q).
\end{align*}
We say that $S$ is well-distributed if for all $t \in \N$, $Q \leq x < x+y \leq 2Q,$ and $(k,\ell)=1$ we have
\begin{align*}
|\{q \in S_t: x \leq q \leq x+y, q \equiv \ell \, (k)\}| \pprec \frac{|S_t(Q)|y}{k Q} +1.
\end{align*}
\begin{lemma} \label{largesieve} Let $\alpha(m)$ and $\beta(n)$ be divisor bounded functions, supported respectively for $m\sim M$ and $n \sim N$. Let $S$ be a well-distributed set and let $Q \geq 1$. Then
\begin{align*}
\sum_{q \in S(Q)}& \frac{q}{\phi(q)} \sideset{}{^\ast} \sum_{\chi \,\,(q)} \bigg| \sum_{m,n} \alpha(m)\beta(n) \chi(mn)\bigg| \\
& \pprec M^{1/2} N^{1/2} ( M + Q M^{1/2} + Q |S(Q)| )^{1/2}( N + Q N^{1/2} + Q |S(Q)| )^{1/2}.
\end{align*}
\end{lemma}

\subsection{Reduction to exponential sums}
In this section we apply Linnik's dispersion method to prove Proposition \ref{typeii}. For a heuristic explanation of the argument that follows we refer to \cite[Section 5.2]{polymath}. Our argument and notations follow closely the proof of \cite[Theorem 5.1(ii)]{polymath}. We may assume that $N \leq X^{1/2-2\varpi-\delta},$ and write $N=X^{1/2-\gamma}$ for $\gamma \in [2 \varpi +\delta,\sigma].$ We then choose $K_0 \leq K$ such that for $R:=P^{K_0},$ $Q:=P^{K-K_0},$ $RQ=D$ we have
\begin{align}\label{RQ1}
R\sqrt{Q} \in [NX^{-2 \delta}, NX^{-\delta/2}].
\end{align}
This holds if
\begin{align} \label{RQ2}
R \in [X^{1/2-2 \varpi - 2 \gamma -3 \delta}, X^{1/2-2\varpi -2 \gamma - 2 \delta}], \quad \quad Q \in [X^{4 \varpi + 2 \gamma + 2 \delta},X^{4 \varpi + 2 \gamma + 3 \delta}].
\end{align}
Define
\begin{align*}
\RR := \{p_1^2 \cdots p_{K_0}^2: \, p_j \in I_j \} \quad \text{and} \quad \QQ := \{p_{K_0+1}^2 \cdots p_{K}^2: \, p_j \in I_j \} 
\end{align*}
so that for $r^2 \in \RR$ and $q^2 \in \QQ$ we have $r^2 \asymp R$ and $q^2 \asymp Q$. Note that then $(r,q)=1.$

\begin{remark} Note that in \cite{polymath} the factorization is chosen so that $R$ is a bit less than $N$. Since the moduli run over a sparse set, we will need a slightly larger $Q$ to control the diagonal contribution. 
\end{remark}

We apply Perron's formula to remove the weight $\psi(mn/X)$ (cf. \cite[Chapter 3]{harman}, for instance).  We also absorb the conditions $1_{m \sim M}, 1_{n \sim N}$ to the coefficients $\alpha(m), \beta(n)$, so that we need to show
\begin{align}
\sum_{\substack{r^2 \in \RR \\ q^2 \in \QQ}} |\Delta (\alpha \ast \beta; a \, (r^2q^2))| \, \ll \frac{X^{1-\eta}}{\sqrt{RQ}},
\end{align}
where the discrepancy is defined by
\begin{align*}
\Delta (\alpha \ast \beta; a \, (r^2q^2)) :=  \sum_{\substack{n \equiv a \, (r^2q^2)}} (\alpha\ast \beta)(n)  - \frac{1}{\phi(r^2)\phi(q^2)} \sum_{\substack{(n,r^2q^2)=1 }}  (\alpha\ast \beta)(n).
\end{align*}

To simplify the application of the dispersion method we split the discrepancy as
\begin{align*}
\Delta (\alpha \ast \beta; a \, (r^2q^2)) = \Delta_1 (\alpha \ast \beta; a; r^2,q^2) + \Delta_2 (\alpha \ast \beta; a; r^2,q^2),
\end{align*}
where
\begin{align*}
\Delta_1 (\alpha \ast \beta; a; r^2,q^2) &:=  \sum_{\substack{n \equiv a \, (r^2q^2)}} (\alpha\ast \beta)(n)  - \frac{1}{\phi(q^2)} \sum_{\substack{(n,q^2)=1 \\ n \equiv a \, (r^2)}}  (\alpha\ast \beta)(n), \\
\Delta_2 (\alpha \ast \beta; a; r^2,q^2) & :=    \frac{1}{\phi(q^2)} \sum_{\substack{(n,q^2)=1 \\ n \equiv a \, (r^2)}}  (\alpha\ast \beta)(n) - \frac{1}{\phi(r^2)\phi(q^2)} \sum_{\substack{(n,r^2q^2)=1 }}  (\alpha\ast \beta)(n).
\end{align*}

We get a sufficient bound for the sum over $\Delta_2$ after expanding by Dirichlet characters and using Lemma \ref{largesieve}. For this we first need to check that $\RR$ is well-distributed but this is immediate since $\RR$ is a subset of density $\log^{-\mathcal{O}(1)} X$ of the set of all squares $m^2 \asymp R$ and the set of all squares is clearly well-distributed. Hence, by Lemma \ref{largesieve} (if we denote by $p_1^{\epsilon_1} \cdots p_{K_0}^{\epsilon_{K_0}}|r$ the modulus of the character which induces $\chi \,(r)$):
\begin{align*}
& \sum_{\substack{r^2 \in \RR \\ q^2 \in \QQ}} |\Delta_2 (\alpha \ast \beta; a; r^2,q^2)|  \, \pprec \frac{1}{\sqrt{Q}} \max_{q^2 \in \QQ }  \sum_{r^2 \in \RR} \frac{1}{\phi(r^2)} \sum_{\substack{\chi \,\, (r^2)\\ \chi \neq \chi_0}} \bigg| \sum_{(mn,q^2)=1} \alpha(m) \beta(n) \chi(mn) \bigg| \\
& \pprec \frac{1}{R\sqrt{Q}} \max_{q^2 \in \QQ } \sum_{\substack{(\epsilon_j) \in \{0,1,2\}^{K_0} \\ (\epsilon_j) \neq \overline{0} }}  \sum_{j=1}^{K_0} \sum_{p_j \in I_j} \frac{\prod_{j=1}^{K_0} p_j^{\epsilon_j}}{\phi \left(\prod_{j=1}^{K_0} p_j^{\epsilon_j} \right) } \sideset{}{^\ast} \sum_{\chi \, (\prod_{j=1}^{K_0} p_j^{\epsilon_j})} \bigg| \sum_{\substack{(mn,q^2)=1 \\ (mn,\prod_{j=1}^{K_0} p_j)=1}} \alpha(m) \beta(n) \chi(mn) \bigg | \\
& \pprec \frac{ M^{1/2} N^{1/2}}{R \sqrt{Q}}\max_{\substack{(\epsilon_j) \in \{0,1,2\}^{K_0} \\ (\epsilon_j) \neq \overline{0} }} P^{\frac{1}{2} |\{j: \, \, \, \epsilon_j =0 \}|} (M + P^{\frac{1}{2} \sum_j \epsilon_j}M^{1/2} )^{1/2}(N + P^{\frac{1}{2} \sum_j \epsilon_j}N^{1/2} )^{1/2} \\
&\ll \frac{ M^{1/2} N^{1/2}}{R \sqrt{Q}} \max_{\substack{(\epsilon_j) \in \{0,1,2\}^{K_0} \\ (\epsilon_j) \neq \overline{0} }} P^{\frac{1}{2} |\{j: \, \, \, \epsilon_j =0 \}|}( M^{1/2}N^{1/2} + M^{1/2}N^{1/4} P^{\frac{1}{4} \sum_j \epsilon_j} + M^{1/4} N^{1/4}  P^{\frac{1}{2} \sum_j \epsilon_j} ) \\
& \ll \frac{MN}{\sqrt{PRQ} } + \frac{MN^{3/4}}{\sqrt{RQ}} + \frac{M^{3/4} N^{3/4}}{\sqrt{Q}} \ll \frac{X^{1- \eta}}{\sqrt{D}},
\end{align*}
since  $P^{\frac{1}{2} \sum_j \epsilon_j} \leq  R \leq N X^{-\delta/2}$.

To handle $\Delta_1$, define the symmetric discrepancy
\begin{align*}
\Delta_0 (\alpha \ast \beta; a, b_1, b_2; r^2,q^2) :=  \sum_{\substack{n \equiv a \, (r^2) \\ n \equiv b_1 \, (q^2)}} (\alpha\ast \beta)(n)  -\sum_{\substack{n \equiv a \, (r^2) \\ n \equiv b_2 \, (q^2)}} (\alpha\ast \beta)(n) 
\end{align*} 
Then by the Chinese remainder theorem we have
\begin{align*}
\sum_{\substack{r^2 \in \RR \\ q^2 \in \QQ}} |\Delta_1 (\alpha \ast \beta; a;r^2,q^2)| \, \leq \frac{1}{\phi(P_I^2)} \sum_{(b, P_I^2) =1} \sum_{\substack{r^2 \in \RR \\ q^2 \in \QQ}} |\Delta_0 (\alpha \ast \beta; a,a,b; r^2,q^2)| ,
\end{align*}
where $P_I= \prod_{p \in I} p$ for $I = \bigcup_{j=1}^K I_j.$ Hence, our claim follows once we show that for all $b_1$ and $b_2$ with $(b_1b_2, P_I^2)=1$ we have
\begin{align}
\sum_{\substack{r^2 \in \RR \\ q^2 \in \QQ}} |\Delta_0 (\alpha \ast \beta; a,b_1,b_2; r^2,q^2)| \, \ll \frac{X^{1-\eta}}{\sqrt{RQ}}
\end{align}

We rearrange the sum and apply Cauchy-Schwarz to get
\begin{align*}
& \sum_{\substack{r^2 \in \RR \\ q^2 \in \QQ}} |\Delta_0 (\alpha \ast \beta; a,b_1,b_2; r^2,q^2)| \, = \sum_{\substack{r^2 \in \RR \\ q^2 \in \QQ}} c_{q,r} \bigg ( \sum_{\substack{n \equiv a \, (r^2) \\ n \equiv b_1 \, (q^2)}} (\alpha\ast \beta)(n)  -\sum_{\substack{n \equiv a \, (r^2) \\ n \equiv b_2 \, (q^2)}} (\alpha\ast \beta)(n)  \bigg ) \\
& = \sum_{r^2 \in \RR} \sum_{m} \alpha(m) \bigg ( \sum_{q^2 \in \QQ}  c_{q,r} \sum_{mn \equiv a \, (r^2)} \beta(n)  ( 1_{mn \equiv b_1 \, (q^2) } - 1_{mn \equiv b_2 \, (q^2) } )  \bigg) \\
& \pprec R^{1/4}M^{1/2} \bigg (\sum_{r^2 \in \RR} \sum_{m} \psi_M(m) \bigg |  \sum_{q^2 \in \QQ}  c_{q,r} \sum_{mn \equiv a \, (r^2)} \beta(n)  ( 1_{mn \equiv b_1 \, (q^2) } - 1_{mn \equiv b_2 \, (q^2) } ) \bigg|^2 \bigg )^{1/2},
\end{align*}
where $\psi_M(m)$ is a $C^\infty$-smooth majorant to $1_{m \sim M}.$ Thus, we need to show that
\begin{align*}
\sum_{r^2 \in \RR} \sum_{m} \psi_M(m) \bigg |  \sum_{q^2 \in \QQ}  c_{q,r} \sum_{mn \equiv a \, (r^2)} \beta(n)  ( 1_{mn \equiv b_1 \, (q^2) } - 1_{mn \equiv b_2 \, (q^2) } ) \bigg|^2 \, \ll \frac{MN^2X^{-\eta}}{QR^{3/2}}.
\end{align*}
Expanding the square we get 
\begin{align*}
 \Sigma(b_1,b_1)-\Sigma(b_1,b_2)-\Sigma(b_2,b_1)+\Sigma(b_2,b_2),
\end{align*}
where
\begin{align*}
\Sigma(b_1,b_2) :=  \sum_{r^2 \in \RR} \sum_{m} \psi_M(m)   \sum_{q_1^2, q_2^2 \in \QQ}  c_{q_1,r} \overline{c_{q_2,r}}\sum_{\substack{mn_1 \equiv a \, (r^2) \\ mn_2 \equiv a \, (r^2)}} \beta(n_1) \overline{\beta(n_2)}  1_{\substack{mn_1 \equiv b_1 \, (q_1^2) \\ mn_2 \equiv b_2 \, (q_2^2) } } .
\end{align*}
Our claim then follows once we show that
\begin{align} \label{sigmaasymp}
\Sigma(b_1,b_2) = X_0 + \mathcal{O} \bigg( \frac{MN^2X^{-\eta}}{QR^{3/2}}\bigg),
\end{align}
where $X_0$ does not depend on $b_1,b_2.$

We first note that since $a$ is coprime to $r^2,$ also $m,$ $n_1$ and $n_2$ are also coprime to $r^2$. Hence, we obtain $n_1 \equiv n_2 \, (r^2).$ Thus, we can write $n_2 = n_1+ \ell r^2$ for some $0 \leq |\ell| \ll L := N/R.$ Therefore,
 \begin{align*}
 \Sigma(b_1,b_2) = \sum_{r^2 \in \RR} \sum_{0 \leq |\ell| \ll L }  \sum_{q_1^2, q_2^2 \in \QQ}  c_{q_1,r} \overline{c_{q_2,r}}\sum_{\substack{n}} \beta(n) \overline{\beta(n +\ell r^2)}   \sum_{m} \psi_M(m) 1_{\substack{  mn \equiv a \, (r^2) \\  mn \equiv b_1 \, (q_1^2) \\  m(n+\ell r^2) \equiv b_2 \, (q_2^2) } } 
 \end{align*}
 
\begin{remark} The size of $L$ is roughly $X^{\gamma +2 \varpi},$ whereas in \cite{polymath} this is of size $X^\delta$. 
\end{remark}

For $\ell = 0$ (i.e. $n_1 = n_2$) the contribution is bounded by
\begin{align*}
 \sum_{r^2 \in \RR}   \sum_{q_1^2, q_2^2 \in \QQ} &  \sum_{\substack{n}} | \beta(n) |^2  \sum_{m} \psi_M(m) 1_{\substack{  mn \equiv a \, (r^2) \\  mn \equiv b_1 \, (q_1^2) \\  mn \equiv b_2 \, (q_2^2) } } \\
&\pprec \, \sum_{r^2 \in \RR}   \sum_{q_1^2, q_2^2 \in \QQ} \sum_{s \asymp X} 1_{\substack{  s \equiv a \, (r^2) \\  s \equiv b_1 \, (q_1^2) \\  s \equiv b_2 \, (q_2^2) } }  \, \ll \,  X \sum_{r^2 \in \RR} \sum_{q_1^2, q_2^2 \in \QQ}  \frac{1}{r^2[q_1^2,q_2^2]}   \\ & \ll \, \frac{X}{\sqrt{R}} \sum_{q_0^2 \ll Q} \sum_{q_1^2, q_2^2 \asymp Q/q_0^2} \frac{1}{q_0^2q_1^2q_2^2} \ll \,  \frac{X}{\sqrt{RQ}} \ll \frac{MN^2X^{-\eta}}{QR^{3/2}},
\end{align*}
since by (\ref{RQ1}) we have $R\sqrt{Q} \leq N X^{-\delta/2}.$ This is sufficient for (\ref{sigmaasymp}) (note that the main contribution to this error term comes from the diagonal, i.e. the part where $(q_1,q_2)$ is big).

For $\ell \neq 0$ we note that $n$ and $n+\ell r^2$ are coprime to $r^2 q_1^2$ and $r^2 q_2^2$, respectively. Hence, we can write
\begin{align} \label{gamma}
 1_{\substack{  mn \equiv a \, (r^2) \\  mn \equiv b_1 \, (q_1^2) \\  m(n+\ell r^2) \equiv b_2 \, (q_2^2) } }  = 1_{m \equiv \theta \, (r^2 [q_1^2,q_2^2])}
\end{align}
for some residue class $\theta= \theta(b_1,b_2,\ell,n,a)$ modulo $r^2 [q_1^2,q_2^2].$ Write $q_0^2 := (q_1^2,q_2^2).$ Then in the sum over $n$ we have a congruence restriction
\begin{align*}
b_1/n \equiv b_2 /(n+ \ell r^2) \,(q_0^2).
\end{align*}
We let $C(n)=C(n;\ell,r,q_0)$ denote the characteristic function of this congruence. We note that by the coprimality of $q_0^2$ and $r^2b_1$ this is the characteristic function of a union of at most
\begin{align*}
(b_1-b_2,q_0^2,\ell r^2b_1) \leq (q_0^2, \ell)
\end{align*} 
congruence classes.

Applying the Poisson summation formula (Lemma \ref{poisson}) we obtain
\begin{align*}
\Sigma(b_1,b_2) = \Sigma_0 (b_1,b_2) +  \Sigma_1 (b_1,b_2) + \mathcal{O}  \bigg( \frac{MN^2X^{-\eta}}{QR^{3/2}}\bigg),
\end{align*}
where
\begin{align*}
\Sigma_0 (b_1,b_2)& := \sum_{m}\psi_M(m) \sum_{r^2 \in \RR} \sum_{0 < |\ell| \ll L }  \sum_{q_1^2, q_2^2 \in \QQ} \frac{c_{q_1,r} \overline{c_{q_2,r}}}{r^2[q_1^2,q_2^2]}  \sum_{\substack{n}} \beta(n) \overline{\beta(n +\ell r^2)}  C(n), \\
\Sigma_1 (b_1,b_2) & \ll_\epsilon \, 1 + X^{2\epsilon} \widehat{\Sigma}_1(b_1,b_2) \quad \quad \text{for} \\
\widehat{\Sigma}_1(b_1,b_2) &:= \sum_{r^2 \in \RR} \sum_{0 < |\ell| \ll L } \sum_{q_0^2 \ll Q}  \sum_{\substack{q_1^2, q_2^2 \in \QQ/q_0^2 \\ (q_1,q_2)=1}}   \frac{1}{H} \sum_{1 \leq |h| \leq H} \bigg | \sum_{\substack{n}} \beta(n) \overline{\beta(n +\ell r^2)} C(n) e_{r^2 q_0^2q_1^2q_2^2}(\theta h) \bigg|,
\end{align*}
where $H=H(q_0,q_1,q_2,r) := X^{\epsilon}r^2 q_0^2q_1^2q_2^2/M,$ and
\begin{align*}
\QQ/q_0^2 := \{q^2: \, \, q^2 q_0^2 \in \QQ \}.
\end{align*}

We write
\begin{align*}
 \Sigma_0 (b_1,b_2)=X_0 + \Sigma_0' (b_1,b_2),
\end{align*}
where $X_0$ is the part with $q_0=1$ and $\Sigma_0' (b_1,b_2)$ corresponds to $q_0 > 1$. Note that $C(n)$ depends on $b_1,b_2$ only if $q_0>1$, so that $X_0$ is independent of $b_1,b_2$. Also, $q_0 > 1$ implies that $q_0^2 \gg X^{\delta}.$ Hence,
\begin{align*}
\Sigma_0' (b_1,b_2) \, \pprec \frac{M}{\sqrt{R}} \sum_{0< |\ell| \ll L} \sum_{X^\delta \ll q_0^2 \ll Q} \sum_{\substack{q_1^2,q_2^2 \asymp Q/q_0^2 }} \frac{1}{q_0^2 q_1^2q_2^2} \sum_{n} C(n).
\end{align*}
We sum over $\ell$ on the inside, recalling the definition of $C(n)$:
\begin{align*}
 \sum_{0< |\ell| \ll L} 1_{b_1/n \equiv b_2 /(n+ \ell r^2) \,(q_0^2)} \, \ll \,1+ L/q_0^2.
\end{align*}
Hence
\begin{align*}
\Sigma_0' (b_1,b_2) \,& \pprec \frac{M}{\sqrt{R}} \sum_{X^\delta \ll q_0^2 \ll Q} \sum_{q_1^2,q_2^2 \asymp Q/q_0^2} \frac{1}{q_0^2 q_1^2q_2^2} \sum_{n} ( 1+ L/q_0 ^2) \\
& \ll \frac{M N}{\sqrt{RQ}} + \frac{MNLX^{-\eta}}{Q\sqrt{R}} \, \ll \frac{MN^2 X^{-\eta}}{QR^{3/2}},
\end{align*}
where the last bound follows from  $R\sqrt{Q} \leq N X^{-\delta/2}$ and $L=N/R$. Thus, we are done once we show the bound
\begin{align} \label{target}
\widehat{\Sigma}_1(b_1,b_2) \, \ll \, \frac{MN^2 X^{-\eta}}{QR^{3/2}}.
\end{align}
For this we need another application of Cauchy-Schwarz to smooth the coefficients in the sum over $n.$   To do this optimally we first need to split the sum over $q_1^2 \in \QQ/q_0^2$ as follows: for 
\begin{align*}
W :=\frac{M^2X^{-2\delta}}{R^2Q^2}= X^{2\gamma-4\varpi - 2\delta} \geq 1,
\end{align*}
 write $Q/q_0^2=UV$ with $V \in [V_0,V_0 X^{\delta}]$ for
\begin{align*}
V_0 := \max \bigg \{ 1, \,\frac{Q}{q_0^2 W}\bigg \}.
\end{align*}
For any fixed $q_0^2$ we may write
\begin{align*}
\sum_{q_1^2 \in \QQ/q_0^2} = \sum_{u^2 \in \UU} \sum_{v^2 \in \VV},
\end{align*}
 where $\UU= \UU(q_0)$ and $\VV=\VV(q_0)$ are sets such that for any $u^2 \in \UU$ and $v^2 \in \VV$, we have $ u^2 \asymp U,$ $v^2 \asymp V$, and $u^2v^2 \in \QQ/q_0^2$. This factorization is again determined so that the diagonal contribution will be just small enough. By Cauchy-Schwarz we obtain
\begin{align*}
\widehat{\Sigma}_1(b_1,b_2) & = \sum_{r^2 \in \RR} \sum_{0 < |\ell| \ll L }  \sum_{q_0^2 \ll Q} \sum_{\substack{q_1^2, q_2^2 \in \QQ/q_0^2 \\ (q_1,q_2)=1}}   \frac{1}{H} \sum_{1 \leq |h| \leq H} \bigg | \sum_{\substack{n}} \beta(n) \overline{\beta(n +\ell r^2)} C(n) e_{r^2 q_0^2q_1^2q_2^2}(\theta h) \bigg|  \\
&= \sum_{r^2 \in \RR} \sum_{0 < |\ell| \ll L }  \sum_{q_0^2 \ll Q} \sum_{u^2 \in \UU}  \sum_{\substack{n}} \beta(n) \overline{\beta(n +\ell r^2)} C(n)  \\
& \hspace{110pt} \sum_{v^2 \in \VV}  \sum_{ \substack{ q_2^2 \in \QQ/q_0^2\\ (uv,q_2)=1}}   \frac{1}{H} \sum_{1 \leq |h| \leq H}  c(r,q_0,u,\ell,v,q_2, h)  e_{r^2 q_0^2 u^2v^2 q_2^2}(\theta h)  \\
& \pprec \sum_{r^2 \in \RR} \sum_{0 < |\ell| \ll L }  \sum_{q_0^2 \ll Q} \sum_{u^2 \in \UU}  \bigg(\frac{N(q_0^2,\ell)}{q_0^2} \bigg)^{1/2} \bigg( \Sigma(r,q_0,u,\ell, b_1,b_2) \bigg)^{1/2},
\end{align*}
where
\begin{align*}
\Sigma(r,q_0,u,\ell,b_1,b_2) &:= \sum_{\substack{n}} \psi_N(n) C(n) \bigg|  \sum_{v^2 \in \VV}  \sum_{ \substack{q_2^2 \in \QQ/q_0^2 \\ (uv,q_2)=1}}   \frac{1}{H} \sum_{1 \leq |h| \leq H} c(r,q_0,u,\ell,v,q_2, h) e_{r^2 q_0^2 u^2v^2 q_2^2}(\theta h) \bigg|^2 \\
& \leq \sum_{v_1^2,v_2^2 \in \VV}  \sum_{\substack{ q_2^2,s_2^2 \in \QQ/q_0^2\\ (uv_1,q_2)=(uv_2,s_2)=1}}    \frac{1}{H^2} \sum_{1 \leq |h_1|,|h_2| \leq H}  |S_{\ell,r^2,u^2}(h_1,h_2,u^2,v_1^2,v_2^2,q_2^2,s_2^2)|
\end{align*}
for
\begin{align*}
S_{\ell,r^2,u^2} = \sum_{n} C(n) \psi_N(n) \Phi_\ell(n; h_1,r^2,q_0^2,u^2v_1^2,q_2^2)\overline{\Phi_\ell(n; h_2,r^2,q_0^2,u^2v_2^2,s_2^2)}
\end{align*}
and
\begin{align}  \label{bigphi}
\Phi_\ell(n; h,r^2,q_0^2,u^2v_1^2,q_2^2) =&  e_{r^2} \bigg( \frac{ah}{q_0^2u^2v_1^2q_2^2 n}\bigg) e_{q_0^2u^2v_1^2} \bigg( \frac{b_1 h}{r^2 q_2^2 n}\bigg) e_{q_2^2} \bigg( \frac{b_2 h}{r^2 q_0^2u^2v_1^2 (n+\ell r^2)}\bigg) 
\end{align}
 Here we have used the Chinese Remainder Theorem (Lemma \ref{crt}) to expand the definition (\ref{gamma}) of $\theta$. Hence
\begin{align*}
\Phi_\ell(n; h_1,r^2,q_0^2,u^2v_1^2,q_2^2)\overline{\Phi_\ell(n; h_2,r^2,q_0^2,u^2v_2^2,s_2^2)} =\, e_{d_1} \bigg ( \frac{c_1}{n} \bigg ) e_{d_2} \bigg ( \frac{c_2}{n+ \tau} \bigg)
\end{align*}
for $\tau:=\ell r^2$, $d_1 : =r^2q_0^2u^2[v_1^2,v_2^2]$, and $d_2 :=[q_2^2,s_2^2]$, for some integers $c_1$ and $c_2$ independent of $n$. Recall that $C(n)$ is the characteristic function of at most $(q_0^2,\ell)$ residue classes modulo $q_0^2$. Hence, applying the triangle inequality and Lemma \ref{expsum2lemma} we obtain
\begin{align*}
 |S_{\ell,r^2,u^2} | \pprec (q_0^2,\ell)\bigg( N^{1/2}(d_1d_2)^{1/6}X^{\delta/6} + \frac{N}{q_0^2} \frac{(c_1,\delta_1')}{\delta_1'} \frac{(c_2,\delta_2')}{\delta_2'}\bigg) .
\end{align*}
Note that
\begin{align*}
(d_1d_2)^{1/6} \leq (Rq_0^2 U V^2 (Q/q_0^2)^2 )^{1/6} \leq R^{1/6}U^{1/6} V^{1/3} Q^{1/3},
\end{align*}
and $(c_2,\delta_2')/\delta_2' \leq 1$. We have
\begin{align*}
\frac{(c_1, \delta_1')}{\delta_1'} \leq \frac{(c_1,r^2)}{r^2}.
\end{align*}
The $r^2$-component in  $ \Phi_\ell(n; h_1,r^2,q_0^2,u^2v_1^2,q_2^2)\overline{\Phi_\ell(n; h_2,r^2,q_0^2,u^2v_2^2,s_2^2)}$
 is by definition
 \begin{align*}
 e_{r^2} \bigg( \frac{ah_1}{q_0^2u^2v_1^2q_2^2 n} -\frac{ah_2}{q_0^2u^2v_2^2s_2^2 n} \bigg).
 \end{align*}
 Since $(r,aq_0 u v_1 v_2 q_2 s_2)=1$, this implies that
\begin{align*}
(c_1,r^2) = (h_1v_2^2s_2^2-h_2v_1^2q_2^2,r^2).
\end{align*}
Therefore,
\begin{align} \label{expsum2}
| S_{\ell,r^2,u^2} | \, \pprec (q_0^2, \ell) \bigg( N^{1/2}R^{1/6}U^{1/6} V^{1/3}Q^{1/3} X^{\delta/6} + \frac{N}{q_0^2R} (h_1v_2^2s_2^2 - h_2v_1^2q_2^2, r^2) \bigg).
\end{align}
Using this we obtain
\begin{align*}
\Sigma(r,u,q_0,\ell,b_1,b_2)  \pprec \sum_{v_1^2,v_2^2 \in \VV}  \sum_{ q_2^2,s_2^2 \in \QQ/q_0^2}    \frac{1}{H^2} \sum_{1 \leq |h_1|,|h_2| \leq H} (q_0^2, \ell) N^{1/2}R^{1/6}U^{1/6} V^{1/3}Q^{1/3}X^{\delta/6}  \\
+  \frac{N(q_0^2, \ell) }{q_0^2R}  \sum_{v_1^2,v_2^2 \in \VV}  \sum_{ q_2^2,s_2^2 \in \QQ/q_0^2}    \frac{1}{H^2} \sum_{1 \leq |h_1|,|h_2| \leq H} (h_1v_2^2s_2^2 - h_2v_1^2q_2^2, r^2) .
\end{align*}
The first term is bounded by
\begin{align*}
\frac{(q_0^2,\ell)}{q_0^2} N^{1/2} R^{1/6} U^{1/6} V^{4/3} Q^{4/3} X^{\delta/6}.
\end{align*}
In the second term we write $\Delta = h_1v_2^2s_2^2 - h_2v_1^2q_2^2$ to get a bound (using Lemma \ref{gcdsum})
\begin{align*}
\pprec \frac{N(q_0^2, \ell) }{q_0^2R} \frac{1}{H^2}\sum_{0 \leq |\Delta| \ll H V Q/q_0^2} (\Delta,r^2) \sum_{h_1,v_2,s_2} 1 \\
\pprec \frac{N(q_0^2, \ell) }{q_0^2R} \frac{1}{H^2} \bigg(H V Q/q_0^2 +R \bigg)H V^{1/2} Q^{1/2}/q_0 \\
= \frac{(q_0^2,\ell)}{q_0^5} \frac{NV^{3/2}Q^{3/2}}{R} + \frac{(q_0^2,\ell)}{q_0^3} \frac{N V^{1/2} Q^{1/2}}{H},
\end{align*}
so that using $H \gg RQ^2/(Mq_0^2)$ we obtain
\begin{align*}
\Sigma(r,u,q_0,\ell,b_1,b_2)  \pprec  \frac{(q_0^2,\ell)}{q_0} \bigg( N^{1/2} R^{1/6} U^{1/6} V^{4/3} Q^{4/3}X^{\delta/6} + \frac{NV^{3/2}Q^{3/2}}{R}  + \frac{MN V^{1/2} }{RQ^{3/2}} \bigg).
\end{align*}
Hence,
\begin{align} \nonumber
& \widehat{\Sigma}_1(b_1,b_2) \, \pprec  N^{1/2} \sum_{r^2 \in \RR}  \sum_{q_0^2 \ll Q} \sum_{u^2 \in \UU} \sum_{0 < |\ell| \ll L }  \bigg(\frac{(q_0^2,\ell)}{q_0^2} \bigg)^{1/2} \bigg( \Sigma(r,u,q_0, \ell, b_1,b_2) \bigg)^{1/2} \\ \label{sigmahatbound}
&\pprec N^{1/2} \sum_{r^2 \in \RR}  \sum_{q_0^2 \ll Q} \sum_{u^2 \in \UU} \sum_{0 < |\ell| \ll L }  \frac{(q_0^2,\ell)}{q_0} \bigg( N^{1/2} R^{1/6} U^{1/6} V^{4/3} Q^{4/3}X^{\delta/6} + \frac{NV^{3/2}Q^{3/2}}{R}  + \frac{MN V^{1/2} }{RQ^{3/2}}  \bigg)^{1/2}
\end{align}

Recalling the definition of $UV = Q/q_0^2$, we separate the sum into two parts:

\textbf{Sum over $q_0^2 \leq Q/W$:} We have $U \leq W$ and $V \leq X^{\delta} Q/W,$ so that by Lemma \ref{gcdsum} we get a contribution 
\begin{align} \label{mainbound}
&\pprec N^{1/2} R^{1/2} W^{1/2}L  \bigg( \frac{N^{1/2} R^{1/6} W^{1/6} Q^{8/3}X^{3\delta/2}}{W^{4/3}} + \frac{NQ^{3}X^{3\delta/2}}{RW^{3/2}}  + \frac{MN X^{\delta/2}}{RQW^{1/2}}  \bigg)^{1/2} \\ \nonumber
&= \frac{MN^2}{QR^{3/2}}\bigg(  \frac{R^{4+1/6}Q^{14/3}L^2X^{3\delta/2}}{M^2N^{5/2}W^{1/6}} + \frac{R^3Q^5L^2X^{3\delta/2}}{M^2N^2W^{1/2}} + \frac{R^3QW^{1/2}L^2X^{\delta/2}}{MN^2} \bigg)^{1/2}
\end{align}
We have to show that the factor in the brackets is $\ll X^{-\eta}.$ For this, recall that $L=N/R$, $W=M^2X^{-2\delta}R^{-2}Q^{-2},$ which gives a bound
\begin{align}
\label{diagfactor} \frac{R^{13/6}Q^{14/3}X^{3\delta/2}}{M^2N^{1/2}W^{1/6}} + \frac{RQ^5X^{3\delta/2}}{M^2W^{1/2}} + \frac{RQW^{1/2}X^{\delta/2}}{M} \\ \nonumber
 =\frac{R^{15/6}Q^{5}X^{11\delta/6}}{M^{7/3}N^{1/2}} + \frac{R^2Q^6X^{5\delta/2}}{M^3} + X^{-\delta/2} \leq X^{-\eta},
\end{align}
if
 \begin{align*}
  &  \frac{R^{15} Q^{30} X^{11\delta}}{M^{14}N^3} \leq \frac{X^{15/2-30\varpi -30 \gamma -30\delta}X^{120 \varpi + 60\gamma + 90\delta}X^{11\delta} }{X^{14/2+14\gamma}X^{3/2-3\gamma}}\leq X^{-\eta}, \quad \text{and}\\
& \frac{R^2Q^6X^{5\delta/2}}{M^3}  \leq \frac{X^{1-4 \varpi - 4 \gamma -4 \delta}X^{24\varpi + 12 \gamma +18 \delta}X^{5\delta /2}}{X^{3/2 + 3 \gamma}} \leq X^{-\eta},
\end{align*}
which holds for some $\eta >0$ if
\begin{align*}
\begin{cases} &19 \gamma + 90 \varpi + 71 \delta < 1,  \quad \text{and}\\
&10 \gamma +40 \varpi + 33 \delta < 1.
\end{cases}
\end{align*}

\textbf{Sum over $q_0^2 > Q/W$:} we have $U =Q/q_0^2$ and $V =1,$ so that 
\begin{align*}
 \sum_{r^2 \in \RR}  \sum_{q_0^2 > Q/W} \sum_{u^2 \in \UU} \sum_{0 < |\ell| \ll L }  \frac{(q_0^2,\ell)}{q_0} \pprec R^{1/2} W^{1/2} L,
\end{align*}
Thus, by (\ref{sigmahatbound}) we get a total contribution bounded by
\begin{align*}
 N^{1/2} R^{1/2} W^{1/2} L  \bigg( N^{1/2} R^{1/6}Q^{3/2}X^{\delta/6} + \frac{NQ^{3/2}}{R}  + \frac{MN  }{RQ^{3/2}}  \bigg)^{1/2}.
\end{align*}
This is smaller than (\ref{mainbound}) since $Q \geq W$. Hence, the bound is sufficient if
\begin{align} \label{sigma}
19 \sigma + 90 \varpi + 71 \delta < 1,
\end{align}  
which holds for some $\delta >0$ since $\sigma = 1/19.5$ and $\varpi=1/4000$. \qed

\begin{remark} \label{gapremark} Note that for  $M \leq X^{1/2+2\varpi},$ even if we keep both of the sums over $q_1$ and $q_2$ inside the application of Cauchy-Schwarz, the diagonal contribution is too large by a factor of $RQ/M = D/M$ (cf. third term in (\ref{diagfactor}) with $W=1$). For this reason we are unable to obtain Type II information when $M,N \in [X^{1/2-2\varpi}, X^{1/2+2\varpi}]$. 
\end{remark}

\begin{remark} Using Lemma \ref{expsum2lemma} instead of Lemma \ref{expsum1lemma} gives a wider range for the Type II sums. Indeed, ignoring exponents that depend on $\varpi$ and $\delta,$ for $q_0=1$ the size of the modulus is
\begin{align*}
RUV^2 Q^2 \approx X^{1/2-2\gamma} X^{2 \gamma} X^0 X^{4\gamma} = X^{1/2+4\gamma},
\end{align*}
while the length of the sum is $N =X^{1/2-\gamma}.$ We have $N < (X^{1/2+4\gamma})^{2/3}$ if $\gamma > 1/22,$ so that in this range we get a better bound by using the $q$-van der Corput method rather than the P\'olya-Vinogradov bound.
\end{remark}

\bibliography{pmodsquare}
\bibliographystyle{abbrv}

\end{document}